\documentclass[12pt]{article}

\usepackage{graphicx}
\usepackage{latexsym,amssymb}
\usepackage{amsthm}
\usepackage{indentfirst}
\usepackage{amsmath}
\usepackage{color}
\usepackage{xcolor}
\usepackage{fourier}
\usepackage[colorlinks=true,backref=page]{hyperref}
\usepackage[all]{xy}

\usepackage[a4paper,text={160true mm,230true mm},centering,top=35true mm,head=5true mm,headsep=2.5true mm,foot=8.5true mm]{geometry}

\newtheorem{theoremalph}{Theorem}

\newtheorem*{Main Theorem}{Main Theorem}

\newtheorem{Theorem}{Theorem}[section]
\newtheorem*{Theorem A}{Theorem A}
\newtheorem*{Theorem A'}{Theorem A'}
\newtheorem*{Theorem B'}{Theorem B'}

\newtheorem{Definition}[Theorem]{Definition}
\newtheorem{Proposition}[Theorem]{Proposition}
\newtheorem{Lemma}[Theorem]{Lemma}

\newtheorem{Remark}[Theorem]{Remark}

\newtheorem{Remark-numbered}[Theorem]{Remark}
\newtheorem{Corollary}[Theorem]{Corollary}

\newtheorem*{Claim}{Claim}
\newtheorem{Claim-numbered}[Theorem]{Claim}

 \def\NN{{\mathbb N}} 
\def\PP{{\mathbb P}}
 \def\RR{{\mathbb R}}

 \def\ZZ{{\mathbb Z}}

\begin{document}

\title{Strong Positive recurrence for potential and exponential mixing of equilibrium states of surface diffeomorphisms}
	
\author{Chiyi Luo and Dawei Yang\footnote{
	D. Yang  was partially supported by National Key R\&D Program of China (2022YFA1005801), NSFC 12325106, ZXL2024386 and Jiangsu Specially Appointed Professorship. 
	C. Luo was partially supported by NSFC 12501244.}}
\date{}    
	
\maketitle
		
\begin{abstract}
In this paper, we study the strong positive recurrence property for a large class of potentials of $C^{\infty}$ surface diffeomorphisms with positive entropy.
We establish several statistical properties of the corresponding equilibrium states, including exponential decay of correlations and effective intrinsic ergodicity.
\end{abstract}
	   
 \tableofcontents

 \section{Introduction}\label{SEC: Introduction} 
 \subsection{Strong positive recurrence for potentials}
 Let $f: M\rightarrow M$ be a diffeomorphism on a compact, boundaryless $C^{\infty}$ Riemannian manifold $M$. 
 Let $\mathcal{M}(f)$ denote the set of all $f$-invariant probability measures, and let $\mathcal{M}_{e}(f)$ denote the set of all ergodic $f$-invariant probability measures.
 For every continuous function $\phi: M\rightarrow \mathbb{R}$, by the variantional principle, the topological pressure $P(\phi)$ is given by
 $$P(\phi):=P(\phi,f)= \sup \Big\{h_{\mu}(f)+\int \phi ~{\rm d}\mu:~ \mu\in  \mathcal{M}(f) \Big\}.$$
 For each $\mu\in  \mathcal{M}(f)$, we denote $P_{\mu}(\phi):=P_{\mu}(\phi,f)=h_{\mu}(f)+\int \phi ~{\rm d}\mu$ the \textit{metric pressure} of $\mu$. 
 A measure $\mu \in \mathcal{M}(f)$ is called an \textit{equilibrium state} of the potential  $\phi$ for $f$ if $P_{\mu}(\phi)=P(\phi)$.
 The equilibrium states of the potential  $\phi \equiv 0$ is called the \textit{measures of maximal entropy}.
 Note that if $\mu$ is an equilibrium state of the potential  $\phi$, then almost every ergodic component of $\mu$ is itself an ergodic equilibrium state of $\phi$.
 
 As described in \cite{Rue78B} and \cite{Sar09}, equilibrium measures constitute one of the most important classes of invariant measures, providing fundamental information for understanding the dynamical behavior of the system. When we consider the dynamics of flows, the study of some important invariant measures can also be transformed into the study of equilibrium states of the Poincaré map. For example, the study of the measure of maximal entropy of a flow is converted into the study of the equilibrium state of the Poincaré map with respect to the roof function. From this perspective, the study of equilibrium states should become an important subject in the field of dynamical systems. We believe that the main results of this paper can play a certain role in the study of the measure of maximal entropy of flows.
 
 For a topologically mixing Anosov diffeomorphism (or more generally an Axiom A diffeomorphism) on a compact manifold, Sinai \cite{Sin72}, Ruelle \cite{Rue76},  Bowen \cite{Bow07} proved that every H\"{o}lder continuous function admits a unique equilibrium state, and this measure is a Gibbs measure and isomorphic to a Bernoulli scheme.
 The proof is based on the construction of a coding map from a mixing subshift of finite type onto the manifold.
 Moreover, based on the statistical properties of subshifts of finite type, the equilibrium state is exponentially mixing \cite{Rue78B}, satisfies large deviation principles \cite{Kif90}, and effective intrinsic ergodicity holds \cite{Kad15}.

Beyond uniform hyperbolicity, equilibrium states do not necessarily exist. 
For instance, Misiurewicz \cite{Mis73} and Buzzi \cite{Buz14} gave examples showing that equilibrium states of the zero potential may fail to exist for $C^{r}$ diffeomorphisms. 
However, by a result of Newhouse \cite{New89}, based on the Gromov-Yomdin theory \cite{Gro87,Yom87}, the entropy map $\mu \mapsto h_{\mu}(f)$ is upper semi-continuous. 
Consequently, equilibrium states of any continuous potential $\phi$ always exist for $C^{\infty}$ diffeomorphisms.

For $C^{\infty}$ surface diffeomorphisms, although the existence of equilibrium states is known \cite{New89}, and their finiteness for certain classes of potentials has been established recently by the remarkable result of Buzzi-Crovisier-Sarig \cite{BCS22}, their statistical properties remain largely unknown.
Recently, outstanding work by Buzzi-Crovisier-Sarig \cite{BCS25} focused on statistical properties of measures of maximal entropy for surface diffeomorphisms. 
By establishing the mechanism of strong positive recurrence (SPR for short) for measures of maximal entropy for surface diffeomorphisms, they proved the exponential decay of correlations and several other important statistical properties for measures of maximal entropy.

In this paper, we study the strong positive recurrence (SPR) property for a large class of potentials whose expressions are similar to those considered in \cite[Corollary 1.4]{BCS22} and establish statistical properties for surface diffeomorphisms.
For instance, for H\"older potentials and for mixing non-uniformly hyperbolic surface diffeomorphisms, we obtain the following result on the exponential decay of correlations (the $\beta$-norm $\|\cdot\|_{\beta}$ is defined in \eqref{eq:holdernorm}).

\begin{Theorem}\label{Thm:T-ET}
	Let  $f: M\rightarrow M$ be a topologically mixing $C^{\infty}$ surface diffeomorphism with positive topological entropy.  
	For every H\"{o}lder continuous function $\phi: M \rightarrow \mathbb{R}$ with $\sup\phi- \inf \phi < h_{\rm top}(f)$, denote by $\mu$ the unique equilibrium state of $\phi$. 
	Then for each $\beta>0$, there exist $C>0$ and $0<\theta<1$ such that for every $\beta$-H\"{o}lder continuous functions $V_1,V_2:M\rightarrow \mathbb{R}$, one has
	$$\left| \int V_1(x) \cdot V_2(f^{n}(x))~{\rm d}\mu(x)- \int V_1~{\rm d}\mu  \cdot \int V_2~{\rm d}\mu \right|\leq C\cdot \theta^n \cdot \|V_1\|_{\beta} \cdot \|V_2\|_{\beta},~\forall n>0.$$
\end{Theorem}
Note that the uniqueness follows from \cite[Corollary 1.4]{BCS22}, and \cite[Theorem1.1]{BCS25} establishes exponential mixing for the measure of maximal entropy, corresponding to the case $\phi \equiv 0$.
We now recall the definition of strong positive recurrence (SPR) for potentials, as introduced by Buzzi-Crovisier-Sarig \cite{BCS25}.
Let  $f: M\rightarrow M$ be a $C^{1+\alpha}$ diffeomorphism on a compact manifold $M$.
We first recall the definition of Pesin sets.
\begin{Definition}
	For $\chi>0$, $0<\varepsilon \ll \chi$ and $\ell\geq 1$, denote by ${\rm PES}_{\ell}^{\chi,\varepsilon}(f)$ the set of points $x$ with a $Df$-invariant splitting $T_{{\rm Orb}(x)}M=E^u\oplus E^s$ such that for every $k\in \mathbb Z$
	\begin{itemize}
		\item $\|D_{f^k(x)}f^{-n}|_{E^u(f^k(x))}\| \leq \ell e^{|k|\varepsilon}e^{-(\chi-\varepsilon)n},~\|D_{f^k(x)} f^{n}|_{E^s(f^k(x))}\|\leq \ell e^{|k|\varepsilon}e^{-(\chi-\varepsilon)n},~\forall n>0$;
		\item $\sin \angle (E^{u} (f^k(x)), E^{s} (f^k(x)))\geq \ell^{-1}e^{-|k|\varepsilon}$.
	\end{itemize}
\end{Definition}
By the Pesin theory \cite{Pesin76,Pesin77,Pesin772}, for every ergodic hyperbolic measure $\mu$ of $f$, if all the Lyapunov exponents of $\mu$ lie outside the interval $[-\chi,\chi]$, then $\mu( \bigcup_{\ell>0} {\rm PES}_{\ell}^{\chi,\varepsilon}(f) )=1$.
For each continuous function $\phi: M \rightarrow \mathbb{R}$, we define 
$$ {\rm Var}( \phi ):= \sup\big\{|\phi(x)-\phi(y)|:~x\in M,~y\in M \big\}.$$   

\begin{Definition}
	Let  $f: M\rightarrow M$ be a diffeomorphism on a compact manifold $M$, and let $\phi: M \rightarrow \mathbb{R}$ be a continuous function. 
	For each $\chi>0$, we call that $f$ is $\chi$-SPR for $\phi$, if for every $\varepsilon>0$ and every sequence of ergodic measures $\{\mu_n\}_{n>0}$ with $P_{\mu_n}(\phi) \rightarrow P(\phi)$, one has 
	$$  \lim_{\ell \rightarrow \infty}~ \liminf_{n\rightarrow \infty}~ 
	      \mu_n({\rm PES}_{\ell}^{\chi,\varepsilon}(f))=1.$$   
\end{Definition}

\begin{theoremalph}\label{Thm: A}
	Let  $f: M\rightarrow M$ be a $C^{\infty}$ surface diffeomorphism with $h_{\rm top}(f)>0$. 
	For every H\"{o}lder continuous function $\phi: M \rightarrow \mathbb{R}$ with ${\rm Var}( \phi )<h_{\rm top}(f)$, there exists $\chi>0$ such that $f$ is $\chi$-SPR for $\phi$.
\end{theoremalph}

Note that in \cite[Theorem G]{BCS25}, it is shown that if  ${\rm Var}(\phi)$ is sufficiently small, then $f$ is $\chi$-SPR for $\phi$ (also see Remark \ref{Rem:SPRpro}). 
However, no explicit bound on how small ${\rm var}(\phi)$ must be is provided. Theorem~\ref{Thm: A} gives a more precise quantitative bound.
To compare with \cite[Theorem A]{BCS25} at the technical level, we establish the SPR property for a class of equilibrium states, using another route:
\begin{itemize}
\item 
For the proof of Theorem~\ref{Thm: A}, we proceed in three steps.
First, we show, by a geometric argument\footnote{We will work with the topological homoclinic class and quadrilaterals, invoking in particular dynamical Sard’s lemma. However, the quadrilateral argument itself will not be showed here, since the required result may be taken directly from Buzzi-Crovisier-Sarig \cite{BCS22}.}, that the limit of any sequence of ergodic measures whose metric pressures converge to the topological pressure is itself ergodic.
Second, we have the continuity of Lyapunov exponents follows from \cite{BCS22I}.
Finally, by applying the results of \cite{BCS25}, we prove the SPR property for the potential.

\item 
In \cite{BCS25}, they first establish the SPR property for MMEs using the continuity of Lyapunov exponents for any sequence of ergodic measures whose metric entropies converge to the topological entropy.
Second, they lift all measures whose metric entropies are close to the topological entropy to finitely many SPR countable Markov shifts.
Finally, using the SPR property of these Markov shifts, they prove that the limit of any sequence of ergodic measures whose metric entropies converge to the topological entropy must itself be ergodic.
\end{itemize}

However, we believe that the first step in \cite{BCS25} is also valid in our setting by understanding their arguments deeply. 
Moreover, it can also be applied to the $C^r$ case with large entropy (see Theorem~\ref{Thm: ErgodicCr}).

Recall the Oseledec theorem, for every $f$-invariant measure $\mu$, the Lyapunov exponents 
\begin{align*}
	\lambda_1(x):= \lim_{n\rightarrow  +\infty} \frac{1}{n} \log \|D_xf^n\|,~~\lambda_2(x):= \lim_{n\rightarrow  +\infty} -\frac{1}{n} \log \|D_xf^{-n}\|, 
\end{align*}
are well-defined for $\mu$-almost every $x$, and we denote $\lambda^{+}(\mu):=\int \lambda_1(x)~{\rm d}\mu(x)$. 
The next theorem is the counterpart of our first step.
\begin{theoremalph}\label{Thm: Erg}
     Let $f:M\rightarrow M$ be a $C^{\infty}$ surface diffeomorphism with $h_{\rm top}(f)>0$.
     Suppose that $\phi$ is a H\"{o}lder continuous function with ${\rm Var}( \phi )< h_{\rm top}(f)$. 
     Then, for any sequence of ergodic measures $\{\mu_n\}_{n>0}$ with $P_{\mu_n}(\phi) \rightarrow P(\phi)$ and $\mu_n \rightarrow \mu$, one has that $\mu$ is ergodic. Moreover, we have $\lambda^{+}(\mu_n) \to \lambda^{+}(\mu)$.
\end{theoremalph}

By \cite[Theorem A]{BCS22} and the ergodicity of the limit measure, the continuity of Lyapunov exponents follows directly.
This part will be discussed in Section \ref{SEC:Erg} and Section \ref{SEC:CofL}, and Theorem \ref{Thm: ErgodicCr} will states the $C^r$ version of this result.

\subsection{Properties derived from SPR for equilibrium states}
In this subsection, we assume throughout that $f: M \to M$ is a $C^{1+\alpha}$ surface diffeomorphism, 
and that $\phi: M \to \mathbb{R}$ is a H\"{o}lder continuous function such that there exists $\chi>0$ for which $f$ is $\chi$-SPR for $\phi$. 
Then, we have existence and finiteness of the ergodic equilibrium states of $\phi$: 

\begin{theoremalph}\label{Thm: B}	
Assume that $f$ is $\chi$-SPR for a H\"{o}lder continuous function $\phi$. 
Then, there exist at most finitely many ergodic equilibrium states $\mu_1,\dots,\mu_n$ of $\phi$. 
For every $i=1,\cdots,n$, the all Lyapunov exponents of  $\mu_i$ are outside $(-\chi,\chi)$.
Moreover, for each sequence of ergodic measures $\{\nu_n\}_{n>0}$ with $P_{\nu_n}(\phi) \to P(\phi)$ and $\nu_n \to \mu$, we have $\mu$ is an ergodic equilibrium state of $\phi$, and $\lambda^{+}(\nu_n) \to \lambda^{+}(\mu)$.
\end{theoremalph}

Theorem \ref{Thm: B} was first proved in \cite[Theorem B, Theorem C]{BCS25} using the SPR property of Markov coding.
Here we provide an alternative proof based on the homoclinic relation (see Theorem \ref{Thm: SPR-BH}), which does not rely on the SPR countable Markov shift. 
(However, we still implicitly use properties of irreducible Markov shifts to show that different ergodic equilibrium states correspond to different Borel homoclinic classes.)

In the setting of Theorem \ref{Thm: B}. 
For each ergodic equilibrium state $\mu$, one has $\mu$ is hyperbolic, and by \cite[Corollary 3.3]{BCS22} there exists $p:=p_{\mu}\in \mathbb{N}^{+}$ such that every ergodic component $\nu$ of $\mu$ with respect to $f^p$ is mixing, and $(f^p,\nu)$ isomorphic to a Bernoulli scheme. 
For every ergodic equilibrium state $\mu$ of $\phi$, we call the integer $p \geq 1$ the \textit{mixing order} of $\mu$.

As a consequence of strong positive recurrence, we get the following statistical properties for ergodic equilibrium states.
For $0<\beta<1$ and a continuous function $V:M\rightarrow \mathbb{R}$, we define
\begin{equation}\label{eq:holdernorm}
	\|V\|_{\beta}:= \sup_{x\in M} |V(x)|+\sup \left\{\frac{|V(x)-V(y)|}{d(x,y)^{\beta}},~x\neq y,~x,y\in M \right\}.
\end{equation}
If $\|V\|_{\beta}<+\infty$, then we call the function $V$ is $\beta$-H\"{o}lder continuous.

\begin{theoremalph}[Exponential mixing] \label{Thm: Exponential mixing}
Assume that $f$ is $\chi$-SPR for a H\"{o}lder continuous function $\phi$. 
Let $\mu$ be an ergodic equilibrium state of $\phi$ of mixing order $p\geq 1$. 
Then, for each $\beta>0$, there are $C>0$ and $\theta\in (0,1)$ such that for every $\beta$-H\"{o}lder continuous functions $V_1,V_2:M\rightarrow \mathbb{R}$ and every ergodic component $\nu$ of $\mu$ w.r.t. $f^p$, one has
		$$\left| \int V_1(x) \cdot V_2(f^{pn}(x))~{\rm d}\nu(x)- \int V_1~{\rm d}\nu  \cdot \int V_2~{\rm d}\nu \right|\leq C\cdot \theta^n \cdot \|V_1\|_{\beta} \cdot \|V_2\|_{\beta},~\forall n>0.$$
\end{theoremalph}

The following theorem provides a quantitative upper-bound of the distance between ergodic measures supported on a homoclinic class and the corresponding equilibrium state.
\begin{theoremalph} \label{Thm: Effective intrinsic ergodicity}
	Assume that $f$ is $\chi$-SPR for a H\"{o}lder continuous function $\phi$. 
	For every $\beta\in (0,1)$, there is $C>0$ such that for every $\beta$-H\"{o}lder continuous function $V:M\rightarrow \mathbb{R}$ and every invariant measure $\nu$, there exists an equilibrium state $\mu$ such that (if the measure $\nu$ is ergodic, then the measure $\mu$ is also an ergodic equilibrium state)
	$$\left| \int V (x) ~{\rm d}\mu(x)- \int V (x)~{\rm d}\nu(x)  \right|\leq C \cdot \|V\|_{\beta} \cdot \sqrt{P_{\mu}(\phi)-P_{\nu}(\phi) }.$$
	Moreover, we also have  $\left| \lambda^{+}(\mu)-\lambda^{+}(\nu) \right|\leq C\sqrt{P_{\mu}(\phi)-P_{\nu}(\phi)}$.
\end{theoremalph}
The following Corollary~\ref{Cor:EIE} improves Theorem~\ref{Thm: Erg}.
\begin{Corollary}[Effective intrinsic ergodicity]\label{Cor:EIE}
	Let $f:M\rightarrow M$ be a $C^{\infty}$ transitive surface diffeomorphism with positive topological entropy.
	Suppose $\phi:M\to \RR$ is H\"{o}lder continuous with ${\rm Var}( \phi )< h_{\rm top}(f)$. 
	Let $\mu$ be the unique equilibrium state of $\phi$.
	Then for each $\beta>0$,  there is $C>0$ such that for every invariant measure $\nu$ and every $\beta$-H\"{o}lder continuous function $V:M\rightarrow \mathbb{R}$, we have 
		$$\left| \int V (x) ~{\rm d}\mu(x)- \int V (x)~{\rm d}\nu(x)  \right|\leq C \cdot \|V\|_{\beta} \cdot \sqrt{P_{\mu}(\phi)-P_{\nu}(\phi) }$$
    and $\left| \lambda^{+}(\mu)-\lambda^{+}(\nu) \right|\leq C\sqrt{P_{\mu}(\phi)-P_{\nu}(\phi)}$, 
\end{Corollary}

Theorems \ref{Thm: Exponential mixing}, Theorem \ref{Thm:T-ET} will be proved in Section \ref{SEC:EM}, Theorem  \ref{Thm: Effective intrinsic ergodicity} and Corollary \ref{Cor:EIE} will be proved in Section \ref{SEC:EIE}.
\begin{theoremalph}[Stability] \label{Thm: Stability}
	Assume that $f$ is $\chi$-SPR for a H\"{o}lder continuous function $\phi$. 
	Then, there exists $\varepsilon>0$ such that for every H\"{o}lder continuous function $\psi$ with $\sup|\psi|<\varepsilon$, 
	we have $f$ is $\chi$-SPR for $\phi+\psi$. 
	Moreover, the map $t\rightarrow P(\phi+t\psi)$ is real-analytic on $(-1,1)$.
\end{theoremalph}

\begin{theoremalph}[Exponential Tails] \label{Thm:  Exponential Tails}
	Assume that $f$ is $\chi$-SPR for a H\"{o}lder continuous function $\phi$. 
	Then, for every $\varepsilon>0$, there are $\ell>0, C>0$ and $\theta\in (0,1)$ such that for every ergodic equilibrium state $\mu$ of $\phi$ and every $n>0$,
	$$\mu\Big(M-\bigcup_{j=0}^{n-1} f^j( {\rm PES}_{\ell}^{\chi,\varepsilon}(f)) \Big)\leq C \cdot \theta^n.$$
\end{theoremalph}
\begin{Remark}
	{\rm Note that Theorem \ref{Thm: Stability} and Theorem \ref{Thm:  Exponential Tails}, 
		hold for every $C^{\infty}$ surface diffeomorphism $f$ with $h_{\rm top}(f)>0$ and every H\"{o}lder continuous function $\phi$ satisfying ${\rm Var}(\phi)<h_{\rm top}(f)$, because in this case, by Theorem \ref{Thm: A}, there exists $\chi>0$ such that $f$ is $\chi$-SPR for $\phi$.}
\end{Remark}

As stated in \cite[Theorem G]{BCS25}, stability, exponential tails, asymptotic variance (see \cite[Section 11.4]{BCS25}), central limit theorem (see \cite[Section 11.5]{BCS25}) and large deviations (see \cite[Section 11.6]{BCS25}) can be established in a similar way by using the SPR properties of countable Markov shifts together with analogous arguments in Sections \ref{SEC:EM} and \ref{SEC:EIE}. 
These properties of SPR countable Markov shifts are detailed in \cite[Theorem B.20]{BCS25} for stability, \cite[Corollary 7.8]{BCS25} for exponential tails, \cite[Section B.5]{BCS25} for asymptotic variance, \cite[Section C.1]{BCS25} for functional central limit theorem, and \cite[Section B.6]{BCS25} for large deviations.

\section{Proof of Theorem \ref{Thm: A}} 
In this subsection, we assume throughout that $f: M \to M$ is a $C^{r},r\in (1,+\infty]$ surface diffeomorphism with $h_{\mathrm{top}}(f)>0$, and that $\phi: M \to \mathbb{R}$ is a continuous function.
We denote 
$$ \lambda_{\min}(f):=\min \left\{ \lim_{n\rightarrow +\infty} \frac{1}{n} \log~ \sup_{x\in M} \|D_xf^n\|, \lim_{n\rightarrow +\infty} \frac{1}{n} \log~ \sup_{x\in M} \|D_xf^{-n}\| \right\}.$$

\subsection{Borel homoclinic classes and topological homoclinic classes}\label{SEC:HC}
Let ${\rm NUH}^{\#}(f)$ be the set of all points $x$ for which there exist $\chi>0$, $0<\varepsilon\ll \chi$, $\ell>0$ and $\{n_k\}, \{m_k\} \subset \NN$ such that $f^{n_k}(x), f^{-m_k}(x)\in {\rm PES}_{\ell}^{\chi,\varepsilon}(f)$ and $f^{n_k}(x) \rightarrow x$, $f^{m_k}(x) \rightarrow x$ as $k\rightarrow +\infty$.
By the Pesin theory and the Poincar\'{e} recurence theorem, for every ergodic hyperbolic measure $\mu$, one has $\mu ({\rm NUH}^{\#}(f))=1$.
For every $x\in {\rm NUH}^{\#}(f)$, the stable and unstable manifold
\begin{align*}
		&W^s(x):=\Big\{z : \limsup_{n\rightarrow +\infty} \frac{1}{n} \log d(f^{n}(x),f^{n}(z))<0 \Big \} \\
		&W^u(x):=\Big\{z : \limsup_{n\rightarrow +\infty} \frac{1}{n} \log d(f^{-n}(x),f^{-n}(z))<0 \Big\} 
\end{align*}
are injectively immersed $C^r$-submanifolds.

For two ergodic hyperbolic measures $\mu_1$ and $\mu_2$, we say that $\mu_1 \preceq \mu_2$, if  there exist $\Lambda_1$ and $\Lambda_2$ such that $\mu_1(\Lambda_1)>0$, $\mu_2(\Lambda_2)>0$ and for any $x\in\Lambda_1$ and $y\in\Lambda_2$, $W^u(x)$ intersects $W^s(y)$ transversely.
We say that $\mu_1$ is \textit{homoclinically related with} $\mu_2$ ($\mu_1 \overset{h}{\sim} \mu_2$),  if $\mu_1 \preceq \mu_2$ and  $\mu_2 \preceq \mu_1$.  
By \cite[Proposition 2.11]{BCS22}, the homoclinic relation is an equivalence relation for ergodic hyperbolic measures.

Suppose $\mathcal{O}$ is a hyperbolic periodic orbit and $\mu_{\mathcal{O}}$ is the ergodic measure supported on $\mathcal{O}$. 
We say that an ergodic hyperbolic measures $\mu$ is homoclinically related with $\mathcal{O}$ ($\mu \overset{h}{\sim} \mathcal{O}$), if $\mu$ is homoclinically related with $\mu_{\mathcal{O}}$. 
In this case, for $\mu$-almost every $x$,  $W^u(x)$ and $W^s(\mathcal{O})$ have transverse intersections, and $W^s(x)$ and $W^u(\mathcal{O})$ have transverse intersections. 
By Katok's shadowding lemma \cite{Kat80}, one knows that for every ergodic hyperbolic measure $\mu$, there exists a hyperbolic periodic orbit $\mathcal{O}$ that homoclinically related with $\mu$.

Let $\mu$ be an ergodic hyperbolic measure of $f$ and let $\mathcal{O}$ be a hyperbolic periodic orbit that $\mu \overset{h}{\sim} \mathcal{O}$. 
We define the \textit{Borel homoclinc class} of $\mu$ by
$${\rm BH}(\mu):={\rm BH}(\mathcal{O})=\{y \in {\rm NUH}^{\#}(f): W^u(y) \pitchfork  W^s(\mathcal{O})\neq \emptyset,~W^s(y) \pitchfork  W^u(\mathcal{O})\neq \emptyset\}; $$
and define the \textit{topological homoclinc class} of $\mu$ by
$${\rm TH}(\mu):={\rm TH}(\mathcal{O})=\overline{W^u(\mathcal{O}) \pitchfork W^s(\mathcal{O})}. $$
By definition, ${\rm BH}(\mu)$ is a measurable $f$-invariant set, ${\rm TH}(\mu)$ is a compact $f$-invariant set and ${\rm TH}(\mu)=\overline{{\rm BH}(\mu)}$.
For topological homoclinc class and Borel homoclinc class, we have the following result 
(we denote $(\lambda_{\min}(f)/\infty):=0$).
\begin{Lemma}\label{Lem:BTHC}
	Let $f:M\rightarrow M$ be a $C^{r},r\in (1,+\infty]$ surface diffeomorphism with positive topological entropy and $\mu$ be an ergodic hyperbolic measure of $f$.
	For every ergodic hyperbolic measure $\nu$, 
	\begin{enumerate}
		\item [(1)] 	$~\nu({\rm BH}(\mu))=1$ if and only if $\mu \overset{h}{\sim} \nu$;
		\item [(2)]     if  $h_{\mu}(f)>\frac{\lambda_{\min}(f)}{r}$ and $h_{\nu}(f)>\frac{\lambda_{\min}(f)}{r}$, then $~\nu({\rm TH}(\mu))=1$ if and only if $\mu \overset{h}{\sim} \nu$.
	\end{enumerate}
	\begin{proof}
		Let $\mathcal{O}$ be a hyperbolic periodic orbit homoclinically related with $\mu$.
		Suppose that $\nu$ is an ergodic hyperbolic measure. 
		If $\nu({\rm BH}(\mu))=1$, then for $\nu$-almost every $x$ one has 
		$$W^u(x) \pitchfork  W^s(\mathcal{O})\neq \emptyset,~W^s(x) \pitchfork  W^u(\mathcal{O})\neq \emptyset,$$
		which implies that $\nu \overset{h}{\sim} \mathcal{O}$, and hence $\nu \overset{h}{\sim} \mu$.
		This proves Item (1).
		
		Item (2) follows form \cite[Corollary 6.7]{BCS22}. 
		Since $h_{\nu}(f)>\frac{\lambda_{\min}(f)}{r}$,  $h_{\mu}(f)>\frac{\lambda_{\min}(f)}{r}$ and $\mu \overset{h}{\sim} \mathcal{O}$, \cite[Corollary 6.7]{BCS22} shows that $\nu({\rm TH}(\mu))=1$ if and only if $\nu \overset{h}{\sim} \mathcal{O}$, and hence if and only if $\mu \overset{h}{\sim} \nu$.
	\end{proof}
\end{Lemma}

Buzzi-Crovisier-Sarig \cite[Theorem 5.2]{BCS22} proved the finiteness of the number of Borel homoclinic classes for ergodic measures of a $C^{r}$ surface diffeomorphism with large entropy:
\begin{Theorem}\label{Thm:FiniteBH}
Let $f:M\rightarrow M$ be a $C^{r},r\in (1,+\infty]$ surface diffeomorphism. 
Then, for any $\delta>0$,
$$\# \left\{{\rm BH}(\mu): \mu ~\text{is an ergodic measure of }~f~\text{with}~h_{\mu}(f)>\delta+\frac{\lambda_{\min}(f)}{r} \right\}<+\infty.$$
\end{Theorem}
\begin{proof}
	Suppose the conclusion fails.
	Then, by Lemma \ref{Lem:BTHC} there exists a sequence of ergodic measures $\{\mu_n \}_{n>0}$ such that $\mu_n$ not homoclinically related with $\mu_m$ for every $n\neq m$, and  
	$h_{\mu_k}(f)>\frac{\lambda_{\min}(f)}{r}+\delta$ for every $k>0$.	
	However, by \cite[Theorem 5.2]{BCS22}
	$$\limsup_{n\rightarrow +\infty} h_{\mu_n}(f)\leq \frac{\lambda_{\min}(f)}{r},$$
	which contradicts  $h_{\mu_n}(f)>\frac{\lambda_{\min}(f)}{r}+\delta$ for all $n>0$.
	This contradiction gives the proof.
\end{proof}

Since ${\rm TH}(\mu)=\overline{{\rm BH}(\mu)}$, we have the following corollary immediately.
\begin{Corollary} \label{Cor: FTB}
	Let $f:M\rightarrow M$ be a $C^{r},r\in (1,+\infty]$ surface diffeomorphism. 
	Then, for any $\delta>0$,
	$$\# \left\{{\rm TH}(\mu): \mu ~\text{is an ergodic measure of }~f~\text{with}~h_{\mu}(f)>\delta+\frac{\lambda_{\min}(f)}{r} \right\}<+\infty.$$
\end{Corollary}

\subsection{Ergodicity of equilibrium states} \label{SEC:Erg}
Let $\phi: M \to \mathbb{R}$ be a continuous function.
Recall the variance of $\phi$ is defined by ${\rm Var}( \phi ):= \sup\big\{|\phi(x)-\phi(y)|:~x\in M,~y\in M \big\}$ and $P_{\mu}(\phi):= h_{\mu}(f)+ \int \phi ~{\rm d}\mu$.

\begin{Theorem}\label{Thm: Ergodic}
		Let $f:M\rightarrow M$ be a $C^{\infty}$ surface diffeomorphism with $h_{\rm top}(f)>0$.
		Suppose that $\phi$ is a H\"{o}lder continuous function with ${\rm Var}( \phi )< h_{\rm top}(f)$. 
		Then, for any sequence of ergodic measures $\{\mu_n\}_{n>0}$ with $P_{\mu_n}(\phi) \rightarrow P(\phi)$ and $\mu_n \rightarrow \mu$, one has that $\mu$ is ergodic.
		\end{Theorem}

Remark that in the setting of Theorem \ref{Thm: Ergodic}, by the upper semi-continuity of the entropy map, we know that $\mu$ is an equilibrium state of $\phi$. 
\begin{Lemma} \label{Lem: EH}
	Let $f:M\rightarrow M$ be a $C^{r},r\in (1,+\infty]$ surface diffeomorphism with $h_{\rm top}(f)>\frac{\lambda_{\min}(f)}{r}$. 
	Suppose that  $\phi$ is a continuous function with ${\rm Var}( \phi )< h_{\rm top}(f)-\frac{\lambda_{\min}(f)}{r}$.  
	Then, there is $\delta>0$ such that for every equilibrium state $\mu$ of $\phi$, one has $h_{\mu}(f)>\delta+\frac{\lambda_{\min}(f)}{r}$.
\end{Lemma}
\begin{proof}
	Let $\mu$ be a equilibrium state  of $\phi$. 
	For $\varepsilon\in \Big(0,\frac{1}{2}\big(h_{\rm top}(f)-\frac{\lambda_{\min}(f)}{r}-{\rm Var}(\phi) \big) \Big)$, take an $f$-invariant measure $\nu$ such that $h_{\nu}(f)\geq h_{\rm top}(f)-\varepsilon$.
	Then, we have 
	\begin{align*}
		h_{\mu}(f) + \int \phi ~{\rm d}\mu 
		&\geq  h_{\nu}(f) + \int \phi~ {\rm d}\nu  \\
		&\geq  h_{\rm top}(f)-\varepsilon + \int \phi~ {\rm d}\nu.
	\end{align*}
	Hence, one has
	\begin{align*}
		h_{\mu}(f)&\geq  h_{\rm top}(f)-\varepsilon-\big(\int \phi ~{\rm d}\mu-\int \phi ~{\rm d}\nu \big) \\
		&\geq h_{\rm top}(f)-\varepsilon-{\rm Var}(\phi) \\
		&\geq \frac{\lambda_{\min}(f)}{r}+\frac{1}{2}\Big(h_{\rm top}(f)-\frac{\lambda_{\min}(f)}{r}-{\rm Var}(\phi) \Big) \\
		:&=\frac{\lambda_{\min}(f)}{r}+\delta>0.
	\end{align*}
	This completes the proof of the Lemma.
\end{proof}	

The following lemma shows the uniqueness of hyperbolic ergodic equilibrium states supported on the Borel homoclinic class.

\begin{Lemma}[\cite{BCS22}, Corollary 3.3] \label{Lem: HES}
		Let $f:M\rightarrow M$ be a $C^{r},r>1$ surface diffeomorphism. 
		Suppose that  $\phi$ is a H\"{o}lder continuous function and $\mu$ is an ergodic, hyperbolic equilibrium state of $\phi$. 
		Then, for every ergodic, hyperbolic equilibrium state $\nu$ of $\phi$, $\nu( {\rm BH}(\mu) )=1$ if and only if $\nu=\mu$. 
\end{Lemma}

\begin{Corollary}\label{Cor:Finte}
	Let $f:M\rightarrow M$ be a $C^{r},r\in (1,+\infty]$ surface diffeomorphism with $h_{\rm top}(f)>\frac{\lambda_{\min}(f)}{r}$. 
	Suppose that  $\phi$ is a  H\"{o}lder continuous function with ${\rm Var}( \phi )< h_{\rm top}(f)-\frac{\lambda_{\min}(f)}{r}$.  
	Then, there exist at most finitely many ergodic equilibrium states.
\end{Corollary}

Corollary \ref{Cor:Finte} follows from Theorem \ref{Thm:FiniteBH}, Lemma  \ref{Lem: EH} and Lemma \ref{Lem: HES} immediately.
We now give the proof of Theorem \ref{Thm: Ergodic}.
\begin{proof}[Proof of Theorem \ref{Thm: Ergodic}]
	Suppose that $\{\mu_n\}_{n>0}$ is a sequence of ergodic measures satisfies that 
	$P_{\mu_n}(\phi) \rightarrow P(\phi)$ and $\mu_n \rightarrow \mu$. 
	We claim that $\lim\limits_{n\rightarrow +\infty} h_{\mu_n}(f)=h_{\mu}(f)$.
	Indeed, by the continuity of $\phi$ we have
	$\lim\limits_{n\rightarrow +\infty} \int \phi ~{\rm d}\mu_n= \int \phi ~{\rm d}\mu$,
	and by assumption we have 
	$$\lim_{n\rightarrow +\infty} h_{\mu_n}(f)+\int \phi ~{\rm d}\mu_n= h_{\mu}(f)+\int \phi ~{\rm d}\mu.$$
	Hence the claim follows.
	By Lemma \ref{Lem: EH}, there exists $\delta>0$ such that $h_{\mu}(f)>\delta$.
	Hence, for $n$ sufficiently large we also have $h_{\mu_n}(f)>\delta$.
	By Corollary \ref{Cor: FTB}, one has 
	$$ \# \left\{{\rm TH}(\mu): \mu ~\text{is an ergodic measure of }~f~\text{with}~h_{\mu}(f)>\delta \right\}<+\infty.$$
	Without loss of generality, by passing to a subsequence, we may assume that  $h_{\mu_n}(f)>\delta$ and ${\rm TH}(\mu_n)={\rm TH}(\mu_1)$ for every $n>0$.
	Since ${\rm TH}(\mu_1)$ is a compact $f$-invariant set and $\mu_n({\rm TH}(\mu_1))=1$ for every $n>0$, it follows that $\mu( {\rm TH}(\mu_1))=1$.
	
	Suppose that $\nu$ is an ergodic component of $\mu$ which supported on ${\rm TH}(\mu_1)$ and is an ergodic equilibrium state of $\phi$. 
	We now prove that almost every ergodic component of $\mu$ coincides with $\nu$, and hence $\mu = \nu$ is an ergodic equilibrium state of $\phi$.
	
	Note that almost every ergodic component $\nu'$ of $\mu$ is supported on ${\rm TH}(\mu_1)$ and is an ergodic equilibrium state of $\phi$.
	Moreover, by Lemma \ref{Lem: EH} we have  $h_{\nu'}(f)>\delta$.
    By Lemma \ref{Lem:BTHC}, since $h_{\nu}(f)>0$, $h_{\nu'}(f)>0$, $\nu({\rm TH}(\mu_1))=1$ and  $\nu'({\rm TH}(\mu_1))=1$, it follows that $\nu' \overset{h}{\sim} \nu$.
	Hence, by Lemma \ref{Lem: HES}, we have $\nu=\nu'$.
	Therefore, we prove that almost every ergodic component of $\mu$ equal to $\nu$.
	This completes the proof of Theorem \ref{Thm: Ergodic}.
\end{proof}

\subsection{Continuity of Lyapunov exponents for  equilibrium states} \label{SEC:CofL}
Let $f$ be a $C^{\infty}$ surface diffeomorphism,
by \cite[Theorem A]{BCS22I}, if a sequence of ergodic measures $\{\mu_n\}$ satisfies $\mu_n \rightarrow \mu$, where $\mu$ is an ergodic measure and $h_{\mu_n}(f) \rightarrow h_{\mu}(f)>0$, then we have $\lambda^{+}( \mu_n) \rightarrow \lambda^{+}(\mu)$.
As a consequence, we have the following corollary of Theorem \ref{Thm: Ergodic}, which completes the proof of Theorem \ref{Thm: Erg}. 
\begin{Corollary}
	In the setting of Theorem \ref{Thm: Ergodic}, for every sequence of ergodic measures $\{\mu_n\}_{n>0}$ with $P_{\mu_n}(\phi) \rightarrow P(\phi)$ and $\mu_n \rightarrow \mu$, one has $\mu$ is an ergodic equilibrium state of $\phi$ for $f$, and $\lambda^{+}(\mu_n)\rightarrow \lambda^{+}(\mu)$.
\end{Corollary}
In this section, we discuss the $C^r$ case of Theorem \ref{Thm: Erg}. 
We will prove the following theorem:
\begin{Theorem}\label{Thm: ErgodicCr}
	Let $f:M\rightarrow M$ be a $C^{r},r>1$ surface diffeomorphism with $h_{\rm top}(f)>\frac{\lambda_{\min}(f)}{r}$.
	Suppose that $\phi$ is a H\"{o}lder continuous function with ${\rm Var}( \phi )< h_{\rm top}(f)-\frac{\lambda_{\min}(f)}{r}$. 
	Then, for any sequence of ergodic measures $\{\mu_n\}_{n>0}$ with $P_{\mu_n}(\phi) \rightarrow P(\phi)$ and $\mu_n \rightarrow \mu$, one has $\mu$ is an equilibrium state of $\phi$ for $f$; $\mu$ is ergodic and $\lambda^{+}(\mu_n)\rightarrow \lambda^{+}(\mu)$.
\end{Theorem}
\begin{proof}
	Without loss of generality, we may assume that 
	$$\lambda_{\min}(f)=\lambda^+(f):= \lim_{n\rightarrow +\infty} \frac{1}{n} \log \sup_{x\in M} \|D_xf^n\|,$$
	otherwise, it suffices to consider $f^{-1}$.
	Assume that $\{\mu_n\}_{n>0}$ is a sequence of ergodic measures with $P_{\mu_n}(\phi) \rightarrow P(\phi)$ and $\mu_n \rightarrow \mu$, we first show that $\mu$ is an equilibrium state of $\phi$.
	As in the proof of Lemma \ref{Lem: EH}, there exists $\delta>0$ such that  
	$$\lim_{n\rightarrow \infty} h_{\mu_n}(f)=P(\phi)-\int \phi~{\rm d}\mu\geq h_{\rm top}(f)-{\rm Var}(\phi)>\frac{\lambda^{+}(f)}{r}+\delta>0.$$
	Therefore, by Ruelle's inequality we can assume that $\mu_n$ is an ergodic hyperbolic measure for every $n>0$.
	For each $n>0$, denote by $\hat{\mu}^{+}_n:= \int \delta_{(x,E^u(x))}~{\rm d}\mu_n(x)$ the lift of $\mu_n$ to $\mathbb{P}TM$, which is supported on unstable bundle.
	Note that $\lambda^{+}(\mu_n)=\int \log \|D_xf|_{E} \| ~{\rm d}\hat{\mu}^{+}_n(x,E)$.
	
	Take $\frac{\lambda^{+}(f)}{r}<C<\frac{\lambda^{+}(f)}{r}+\delta$, by \cite[Main Theorem]{Bur24P}, there exists $\beta\in [0,1]$ such that 
	$$\hat{\mu}^{+}_n \rightarrow \hat{\mu}= \beta \hat{\nu}_1^{+}+(1-\beta) \hat{\nu}_0,~\text{and}~\lim_{n\rightarrow \infty} h_{\mu_n}(f)\leq \beta h_{\nu_1}(f)+(1-\beta) \cdot C,$$
    where $\nu_1$ is a projection of $\hat{\nu}_1^{+}$, $\nu_0$ is a projection of $\hat{\nu}_0$ and $\lambda^{+}(\nu_1)=\int \log \|D_xf|_{E} \| ~{\rm d}\hat{\nu}^{+}_1(x,E)$.
    Therefore, we have
    \begin{equation}\label{eq:Esta}
    	P(\phi)=\lim_{n\rightarrow \infty} h_{\mu_n}(f) +\int \phi ~{\rm d}\mu\leq \beta\big(h_{\nu_1}(f) +\int \phi ~{\rm d}\nu_1\big)+(1-\beta) \big(\int \phi ~{\rm d}\nu_0+C\big).
    \end{equation}
    Note that for every $f$-invariant measure $m$, one has
    \begin{align*}
    	P(\phi) &\geq h_{m}(f)+\int \phi ~{\rm d}m \\
    	&=h_{m}(f)+\int \phi ~{\rm d}\nu_0+\int \phi ~{\rm d}m-\int \phi ~{\rm d}\nu_0 \\
    	&\geq h_{m}(f)+\int \phi ~{\rm d}\nu_0-{\rm Var}(\phi).
    \end{align*}
   Therefore, by variational  principle we have 
   $$P(\phi)\geq h_{\rm top}(f)+\int \phi ~{\rm d}\nu_0-{\rm Var}(\phi) > \dfrac{\lambda^{+}(f)}{r}+\delta +\int \phi ~{\rm d}\nu_0>C+\int \phi ~{\rm d}\nu_0.$$
   By $P(\phi) \geq h_{\nu_1}(f)+\int \phi ~{\rm d}\nu_1$ and \eqref{eq:Esta}, we have 
   $$P(\phi)\leq \beta \cdot P(\phi)+(1-\beta)\big(\int \phi ~{\rm d}\nu_0+C\big),~\text{and}~\big(\int \phi ~{\rm d}\nu_0+C\big)<P(\phi).$$
   Then we must have $\beta=1$, $\mu=\nu_1$ and $P(\phi)= h_{\mu}(f)+\int \phi ~{\rm d}\mu$.
   This proves that $\mu$ is an equilibrium state, and 
   $$\lim_{n\rightarrow \infty} \lambda^{+}(\mu_n)=\lim_{n\to \infty}\int \log \|D_xf|_{E} \| ~{\rm d}\hat{\mu}^{+}_n(x,E)= \int \log \|D_xf|_{E} \| ~{\rm d}\hat{\nu}^{+}_1(x,E)= \lambda^{+}(\nu_1)=\lambda^{+}(\mu).$$
   
   To complete the proof of the theorem, it remains to show that $\mu$ is an ergodic measure.
   This discussion is similar to the proof of Theorem \ref{Thm: Ergodic}.
   By Corollary \ref{Cor: FTB}, one has 
   $$ \# \left\{{\rm TH}(\mu): \mu ~\text{is an ergodic measure of }~f~\text{with}~h_{\mu}(f)>\frac{\lambda_{\min}(f)}{r}+\delta \right\}<+\infty.$$
   Without loss of generality, by passing to a subsequence, we may assume that  $h_{\mu_n}(f)>\delta+\frac{\lambda_{\min}(f)}{r}$ and ${\rm TH}(\mu_n)={\rm TH}(\mu_1)$ for every $n>0$. Therefore, we have $\mu({\rm TH}(\mu_1))=1$. 
   Suppose that $\nu$ is an ergodic component of $\mu$ which supported on ${\rm TH}(\mu_1)$ and is an ergodic equilibrium state of $\phi$.  By Lemma \ref{Lem: EH}, we have $h_{\nu}(f)>\delta+\frac{\lambda_{\min}(f)}{r}$.
   
   Since almost every ergodic component $\nu'$ of $\mu$ is supported on ${\rm TH}(\mu_1)$ and, again by Lemma \ref{Lem: EH}, $\nu'$ is an ergodic equilibrium state of $\phi$ with $h_{\nu'}(f)>\delta+\frac{\lambda_{\min}(f)}{r}$. 
  Therefore, by Lemma \ref{Lem:BTHC} and Lemma \ref{Lem: EH}, it follows from Lemma \ref{Lem:BTHC} and Lemma \ref{Lem: EH} that almost every ergodic component of $\mu$ equals $\nu$. Hence, $\mu$ is an ergodic measure of $f$.
\end{proof}

\subsection{Strong positive recurrence for potentials} \label{SEC:SPR}
For a continuous function $\phi: M \to \mathbb{R}$ and $\chi>0$, recall that $f$ is said to be $\chi$-SPR for $\phi$, if for any $\varepsilon>0$ and any sequence of ergodic measures $\{\mu_n\}_{n>0}$ with $P_{\mu_n}(\phi) \to P(\phi)$, one has
$$\lim\limits_{\ell \rightarrow \infty} \liminf\limits_{n\rightarrow \infty} \mu_n({\rm PES}_{\ell}^{\chi,\varepsilon}(f))=1.$$
Note that this definition equivalent to \cite[Defintion 4.14]{BCS25}: 
$$\forall \tau\in (0,1),~\exists \ell>0,~\exists P_0\in (0, P(\phi))~\text{s.t.}~\forall \mu \in \mathcal{M}_{e}(f),~P_{\mu}(\phi)>P_0 \Rightarrow \mu({\rm PES}_{\ell}^{\chi,\varepsilon}(f))>\tau.$$
In this subsection, we will establish the SPR property for a certain class of potentials.

\begin{Theorem}\label{Thm: SPR}
	Let $f:M\rightarrow M$ be a $C^{r},r\in (1,+\infty]$ surface diffeomorphism with $h_{\rm top}(f)>\frac{1}{r}\cdot \lambda_{\min}(f)$.
	Suppose that $\phi$ is a H\"{o}lder continuous function with ${\rm Var}( \phi )< h_{\rm top}(f)-\frac{1}{r}\cdot \lambda_{\min}(f)$. 
	Then, there exists $\chi>0$ such that $f$ is $\chi$-SPR for $\phi$.
\end{Theorem}
\begin{proof}
	By Lemma \ref{Lem: EH} and Theorem \ref{Thm: ErgodicCr}, there exists $\chi>0$ such that for every sequence of ergodic measures $\{\mu_n\}_{n>0}$ with $\mu_n \to \mu$ and $P_{\mu_n}(\phi) \to P(\phi)$, we have $\lambda^{+}(\mu_n) \to \lambda^{+}(\mu)$. Moreover, since $\mu$ is ergodic and $h_{\mu}(f)>\chi$, it follows that for $\mu$-almost every $x$ one has $\lambda^{+}(x)>\chi>-\chi>\lambda^{-}(x)$, where $\lambda^{+}(x)$ and $\lambda^{-}(x)$ are positive and negative Lyapunov exponents of $x$. 
	
	With the facts established above, the proof of Theorem \ref{Thm: SPR} proceeds in essentially the same way as the proof of  \cite[Theorem 3.1]{BCS25}; see also the discussion in \cite[Remark 3.5]{BCS25}.
\end{proof}
Since Theorem \ref{Thm: A} corresponds to the case $r=\infty$ of Theorem \ref{Thm: SPR}, this completes the proof of Theorem \ref{Thm: A}.

\subsection{Strong positive recurrence on Borel homoclinic classes}\label{SEC:SPR-B}
We now introduce the notion of SPR on a homoclinic class, which was defined in \cite[Section 2.2]{BCS25}.
Recall the definition of ${\rm BH}(\mathcal{O})$ and ${\rm TH}(\mathcal{O})$ in Section \ref{SEC:HC}.
Suppose that $\phi:M \rightarrow \RR$ is a H\"{o}lder continuous function.
For each $f$-invariant measurable set $X$, we define the \textit{top-pressure} of $\phi$ in $X$ by 
$$P_{X}(\phi):= P_{X}(\phi,f)= \sup \Big\{ h_{\mu}(f)+ \int \phi~{\rm d}\mu: \mu\in \mathcal{M}(f)~\text{and}~\mu(X)=1 \Big\},$$
and we denote $h_{X}(f):=P_{X}(0)$.
Clearly, we have that $P_{M}(\phi)=P(\phi)$ and $h_{M}(f):=h_{\rm top}(f)$.
We call a measure $\mu\in \mathcal{M}(f)$ is an \textit{equilibrium state of $\phi$ on $X$}, if $\mu(X)=1$ and $P_{\mu}(\phi)=P_{X}(\phi)$.
For $\chi>0$, we call that \textit{$f$ is $\chi$-SPR for $\phi$ on an invariant measurable set $X$}, if for every $\varepsilon>0$ and every sequence of ergodic measures
$\{\mu_n\}_{n>0}$ with $\mu_n(X)=1~(\forall n>0)$ and $P_{\mu_n}(\phi) \rightarrow P_{X}(\phi)$, one has
$$\lim\limits_{\ell \rightarrow \infty} \liminf\limits_{n\rightarrow \infty} \mu_n({\rm PES}_{\ell}^{\chi,\varepsilon}(f))=1.$$
When we say that $f$ is $\chi$-SPR for $\phi$ without further specification, it is understood that $f$ is $\chi$-SPR for $\phi$ on $M$.

\begin{Theorem}
	Let $f:M\rightarrow M$ be a $C^{r}~,r\in (1,+\infty]$ surface diffeomorphism.
	For every Borel homoclinic class $X$ with $h_{\rm X}(f)>\frac{1}{r}\cdot \lambda_{\min}(f)$, if $~\phi:M \rightarrow \RR$ is a H\"{o}lder continuous function with ${\rm Var}(\phi|_{X} )< h_{X}(f)-\frac{1}{r}\cdot \lambda_{\min}(f)$. 
	Then, there exists $\chi>0$ such that $f$ is $\chi$-SPR for $\phi$ on $X$.
\end{Theorem}
\begin{proof}
Let $\mathcal{O}$ be a hyperbolic periodic orbit such that ${\rm BH}(\mathcal{O})=X$.
Recall that ${\rm TH}(\mathcal{O}) = \overline{{\rm BH}(\mathcal{O})}$. 

We first show that
$h_{X}(f):=\sup \{h_{\mu}(f):~\mu({\rm TH}(\mathcal{O}))=1 \}$.
Since every $\mu$ with $\mu(X)=1$ also satisfies $\mu({\rm TH}(\mathcal{O}))=1$, the inequality “$\leq$” is immediate.
On the other hand, since $h_{X}(f)>\frac{1}{r}\cdot \lambda_{\min}(f)$, there exists an ergodic measure $\mu_{0}$ with
$h_{\mu_0}(f)>\tfrac{1}{r}\cdot \lambda_{\min}(f)$ and  $\mu_0 \overset{h}{\sim} \mathcal{O}$.
For any measure $\nu$ with $\nu({\rm TH}(\mathcal{O}))=1$ and $h_{\nu}(f)>\tfrac{1}{r}\cdot \lambda_{\min}(f)$, Lemma \ref{Lem:BTHC} implies $\mu_0 \overset{h}{\sim} \nu$, hence $\nu(X)=1$.
This proves the inequality “$\geq$”.

Now, let $\{\mu_n\}_{n>0}$ be a sequence of ergodic measures on $X$ such that $\mu_n \to \mu$ and $P_{\mu_n}(\phi)\to P_X(\phi)$.
Arguing as in Theorem \ref{Thm: ErgodicCr}, we have that $h_{\mu}(f)>\tfrac{1}{r}\cdot \lambda_{\min}(f)$, $\lambda^+(\mu_n)\to \lambda^+(\mu)$, and that $\mu$ is ergodic with $\mu(X)=1$ since $\mu\overset{h}{\sim} \nu$ for some $\nu$ with $\nu(X)=1$ and $h_{\nu}(f)>\tfrac{1}{r}\cdot \lambda_{\min}(f)$.
Therefore, the SPR property follows.
\end{proof}

\section{Statistic properties of equilibrium states}
In this section, we assume throughout that $f: M \to M$ is a $C^{1+\alpha}$ surface diffeomorphism with positive entropy and $\phi: M \to \mathbb{R}$ is a H\"{o}lder continuous function. 
We enumerate the statistical properties implied by the SPR condition, similar results for measures of maximal entropy were obtained in \cite{BCS25}.

\subsection{Existence of equilibrium states}\label{SEC:SPR-E}
Recall the definition of SPR on a homoclinic class in Section \ref{SEC:SPR-B}.
The following theorem,  firstly proved in \cite{BCS25}, shows the existence of equilibrium states. 
Here, we present an alternative proof based on the homoclinic relation, without relying on the SPR property of Markov coding.
Note that in the statement, the set  $X$ is not assumed to be compact.
\begin{Theorem}\label{Thm: SPR-BH}
	Assume that $X\subset M$ is a Borel homoclinic class and $f$ is $\chi$-SPR for $\phi$  on $X$.
	Then, for every sequence of ergodic measures $\{\mu_n\}_{n>0}$ supported on $X$ with  $P_{\mu_n}(\phi) \rightarrow P_{X}(\phi)$ and $\mu_n \to \mu$, we have $\mu$ is ergodic, $\mu(X)=1$ and $P_{\mu}(\phi)=P_{X}(\phi)$. 
\end{Theorem}

We will prove this conclusion in three steps.
In the first step, we show that $\lambda^{+}(\mu_n) \rightarrow \lambda^{+}(\mu)$, as a consequence of  \cite[Corollary C]{LuY25} we obtain $P_{\mu}(\phi)\geq P_{X}(\phi)$.
In the second step, we prove that almost every ergodic component of $\mu$ gives full measure on $X$, and hence $\mu(X)=1$ and $P_{\mu}(\phi)=P_{X}(\phi)$.
In the final step, we show that $\mu$ is ergodic.
   
\begin{Proposition}\label{Prop:CLE}
	Assume that $X\subset M$ is a Borel homoclinic class and $f$ is $\chi$-SPR for $\phi$  on $X$. 
	For every sequence of ergodic measures $\{\mu_n\}_{n>0}$ with $\mu_n(X)=1,~\forall n>0$, $P_{\mu_n}(\phi) \rightarrow P_{X}(\phi)$ and $\mu_n \to \mu$ as $n\to \infty$, we have $\lambda^{+}(\mu_n) \to \lambda^{+}(\mu)$ as $n\to \infty$.
\end{Proposition}
\begin{proof}   
   For each $\tau>0$, by $\chi$-SPR property there are $\chi>0, ~0< \varepsilon \ll \chi, ~\ell>0$ and $N>0$ such that 
   $$\mu_n({\rm PES}_{\ell}^{\chi,\varepsilon}(f))>1-\tau,~\forall n>N. $$
   Since the set ${\rm PES}_{\ell}^{\chi,\varepsilon}(f)$ is compact, one has $\mu({\rm PES}_{\ell}^{\chi,\varepsilon}(f))>1-\tau$.
   Since $\tau>0$ is arbitrary, we have $\mu(\bigcup_{\ell>0}{\rm PES}_{\ell}^{\chi,\varepsilon}(f))=1$, which implies $\mu$ is  $(\chi-\varepsilon)$-hyperbolic.
   
   Note that the geometric potential $J^u(x):=\log \| D_x f|_{E^u(x)} \|$ is continuous on ${\rm PES}_{\ell}^{\chi,\varepsilon}(f)$.
   Therefore, we can choose a continuous function $\hat{J}: M\rightarrow \RR$ such that 
   $\| \hat{J} \|_{C^0} \leq \log \| Df \|_{\sup}$ and $\hat{J}(x)=J^u(x)$ for every $x\in {\rm PES}_{\ell}^{\chi,\varepsilon}(X)$, where $ \| Df \|_{\sup}:=\sup_{x\in M} \|D_xf\|$.
   Therefore, we have 
   \begin{align*}
   	\limsup_{n\rightarrow \infty} |\lambda^{+}(\mu_n)-\lambda^{+}(\mu)|&= \limsup_{n\rightarrow \infty} \Big| \int  J^u(x) {\rm d} \mu_n(x)- \int  J^u(x) {\rm d} \mu(x) \Big| \\
   	&\leq \lim_{n\rightarrow \infty} \Big| \int  \hat{J}~{\rm d} \mu_n- \int  \hat{J}~{\rm d} \mu \Big|+4\tau \cdot \log  \|Df\|_{\sup}
   	=4\tau \cdot \log \|Df\|_{\sup}.
  \end{align*}
  Since $\tau>0$ is arbitrary, we have $\lambda^{+}(\mu_n) \rightarrow \lambda^{+}(\mu)$. 
\end{proof}

\begin{proof}[Proof of Theorem \ref{Thm: SPR-BH}] 
  {\it Step 1.}  
  By the continuity of $\phi$, we have
  $$\lim_{n\rightarrow \infty} \int \phi ~{\rm d}\mu_n= \int \phi ~{\rm d}\mu.$$
  To show that $P_{\mu}(\phi)\geq P_{X}(\phi)$, it remains to establish the upper semi-continuity of the entropy.
  By \cite[Corollary C]{LuY25},  if $\lambda^{+}(\mu_n) \rightarrow \lambda^{+}(\mu)$, then we have 
  $$\limsup\limits_{n\rightarrow +\infty} h_{\mu_n}(f)\leq h_{\mu}(f).$$
  Since Proposition \ref{Prop:CLE} ensures that $\lambda^{+}(\mu_n) \rightarrow \lambda^{+}(\mu)$, the desired inequality $P_{\mu}(\phi)\geq P_{X}(\phi)$ follows.
 
   \medskip
   
  \noindent {\it Step 2.} 
  Let $\mathcal{O}$ be a hyperbolic periodic orbit such that ${\rm BH}(\mathcal{O})=X$. 
  From the first step, we know that $\mu$ is a hyperbolic measure.
  We now prove that almost every ergodic component $\nu$ of $\mu$ satisfies $\nu \overset{h}{\sim} \mathcal{O}$.
  Consider the ergodic decomposition of $\mu=\int \mu_x ~{\rm d}\mu(x)$, where $\mu_x=\lim\limits_{n\rightarrow \infty} \frac{1}{n} \sum_{j=0}^{n-1} \delta_{f^j(x)}  $ whenever the limit exists. Let 
  $$E:=\{x: \mu_x ~\text{is well defined and}~\mu_x({\rm BH}(\mathcal{O}))=0\}.$$

  It suffices to prove that $\mu(E)=0$.
  Suppose, for contradiction, that $\mu(E)=\tau>0$.
  Since $E$ is an $f$-invariant measurable subset, we can decompose $\mu=\tau  \mu_{\ast}+(1-\tau)\mu_0$, where $\mu_{\ast}:= \frac{1}{\tau} \int_{E} \mu_x ~{\rm d}\mu(x)$ is an $f$-invariant probability measure. 
  By $\chi$-SPR property, there exist $\ell>0$ and $N>0$ such that 
  $$\mu_n ({\rm PES}_{\ell}^{\chi,\varepsilon}(f))>1-\frac{\tau^2}{4},~~\forall n>N.$$
  It follows that $\mu({\rm PES}_{\ell}^{\chi,\varepsilon}(f))>1-\frac{\tau^2}{3}$. 
  Hence, we have 
  $$\mu_{\ast}({\rm PES}_{\ell}^{\chi,\varepsilon}(f))>\frac{1}{\tau} \left( \mu({\rm PES}_{\ell}^{\chi,\varepsilon}(f))-(1-\tau)\mu_0({\rm PES}_{\ell}^{\chi,\varepsilon}(f)) \right)>\frac{1}{\tau} \left(1-\frac{\tau^2}{3}-1+\tau \right)=1-\frac{\tau}{3}.$$
  
  Take $\delta_{\ell}>0$ such that for every $x_1,x_2\in {\rm PES}_{\ell}^{\chi,\varepsilon}(f)$ with $d(x_1,x_2)\leq 2\delta_{\ell}$, one has $W^u_{\rm loc}(x_1) \pitchfork W^s_{\rm loc}(x_2) \neq \emptyset $ and 
  $W^s_{\rm loc}(x_1) \pitchfork W^u_{\rm loc}(x_2) \neq \emptyset $.
  Choose a finite set $\{z_1,\cdots, z_k\} \subset {\rm PES}_{\ell}^{\chi,\varepsilon}(f)$ such that 
  $$\mu_{\ast}\Big({\rm PES}_{\ell}^{\chi,\varepsilon}(f) - \bigcup_{j=1}^{k} B(z_j,\delta_{\ell})\Big)=0
  ~\text{and}~\mu_{\ast}\big(B(z_i,\delta_{\ell})\cap {\rm PES}_{\ell}^{\chi,\varepsilon}(f)\big)>0,~\forall 1\leq i \leq k.$$
  Then, we have 
  $$\mu\big( \bigcup_{j=1}^{k} B(z_j,\delta_{\ell})\big)\geq \tau \cdot \mu_{\ast}({\rm PES}_{\ell}^{\chi,\varepsilon}(f))>\tau-\frac{\tau^2}{3}>\frac{2\tau}{3}.$$
  By enlarging  $N>0$ if necessary, we may assume that 
  $$\mu_n \big( \bigcup_{j=1}^{k} B(z_j,\delta_{\ell}) \big)>\frac{\tau}{2},~~\forall n>N.$$
  Fix some $n>N$. Then, $\mu_n \big( \bigcup_{j=1}^{k} B(z_j,\delta_{\ell})\cap  {\rm PES}_{\ell}^{\chi,\varepsilon}(f) \big)>0$, so there exists $i\in \{1,\cdots,k\}$ such that 
  $$\mu_n\big( B(z_i,\delta_{\ell})\cap  {\rm PES}_{\ell}^{\chi,\varepsilon}(f) \big)>0.$$
  Since $\mu_{\ast}(B(z_i,\delta_{\ell})\cap {\rm PES}_{\ell}^{\chi,\varepsilon}(f))>0$, there is $x\in E$ such that $\mu_{x}({\rm PES}_{\ell}^{\chi,\varepsilon}(f)\cap B(z_i,\delta_{\ell}))>0$. 
  Now, for every point $y_1,y_2\in B(z_i,\delta_{\ell})\cap {\rm PES}_{\ell}^{\chi,\varepsilon}(f)$, we have
  $W^u_{\rm loc}(y_1) \pitchfork W^s_{\rm loc}(y_2) \neq \emptyset $ and 
  $W^s_{\rm loc}(y_1) \pitchfork W^u_{\rm loc}(y_2) \neq \emptyset $. 
  This implies that $\mu_n \overset{h}{\sim} \mu_x$, and consequently $\mu_x \overset{h}{\sim} \mathcal{O}$.
  However, this contradicts the assumption that $\mu_x( {\rm BH}(\mathcal{O}))=0$ for every $x\in E$.
  Therefore, we must have $\mu(E)=0$.  
  
  Since $\mu(E)=0$, it follows that for $\mu$-almost every $x$ one has $\mu_x({\rm BH}(\mathcal{O}))=1$.
  Consequently, $\mu (X)=1$. 
  By the definition of $P_X(\phi)$, one has $P_{\mu}(\phi)\leq P_X(\phi)$. 
  Together with the conclusion from the first step, this yields $P_{\mu}(\phi)=P_X(\phi)$ and $P_{\mu_x}(\phi)=P_X(\phi)$ for $\mu$-almost every $x$. 
  
  \medskip
  
  \noindent {\it Step 3.} 
  By \cite[Section 3.5]{BCS22}, if $\nu$ is an ergodic hyperbolic equilibrium state of $\phi$ on $X$, then every ergodic hyperbolic equilibrium state $\nu’$ for $\phi$ on $X$ coincides with $\nu$.
  From the first and second steps, we know that $\mu$ is a hyperbolic invariant equilibrium state of $\phi$ on $X$, and that almost every ergodic component of $\mu$ is an ergodic hyperbolic equilibrium state of $\phi$ on $X$.
  It follows that almost all ergodic components of $\mu$ coincide.
  Therefore, $\mu$ itself is an ergodic hyperbolic equilibrium state of $\phi$.
 \end{proof}

 \begin{proof}[Proof of Theorem \ref{Thm: B}]
 	We now assume that $f: M\rightarrow M$ is a $C^{1+\alpha}$ surface diffeomorphism, and that $\phi: M \to \mathbb{R}$ is a H\"{o}lder continuous function such that there exists $\chi>0$ for which $f$ is $\chi$-SPR for $\phi$. 
 	Then, there exist $\tau>0$, $\ell>0$ and $P_0\in (0,P(\phi))$ such that for every ergodic measure $\mu$, if $P_{\mu}(\phi)>P_0$, then $\mu({\rm PES}_{\ell}^{\chi,\varepsilon}(f))>\tau$.
 	Take $\delta_{\ell}>0$ such that for every $x_1,x_2\in {\rm PES}_{\ell}^{\chi,\varepsilon}(f)$ with $d(x_1,x_2)\leq 2\delta_{\ell}$, one has $W^u_{\rm loc}(x_i) \pitchfork W^s_{\rm loc}(x_j) \neq \emptyset$, $i,j=1,2$.
 	Choose finite points $\{z_1,\cdots, z_k\} \subset {\rm PES}_{\ell}^{\chi,\varepsilon}(f)$ such that $\cup_{i=1}^{k}B(z_i,\delta_{\ell})\supset {\rm PES}_{\ell}^{\chi,\varepsilon}(f)$. 	
 	Let $z\in M$, if two ergodic hyperbolic measures $\mu_1$ and $\mu_2$ satisfying  $\mu_i \big( B(z,\delta) \cap{\rm PES}_{\ell}^{\chi,\varepsilon}(f)\big)>0,~i=1,2$,
 	then $ \mu_1 \overset{h}{\sim} \mu_2$. 
 	Consequently, 
 	\begin{equation}
 	\# \{{\rm BH}(\mu): \mu~\text{is ergodic hyperbolic and}~P_{\mu}(\phi)>P_0\}\leq k<+\infty.
 	\end{equation}
 	Therefore, the number of ergodic equilibrium states is at most finite. 
 	Let $\{\mu_n\}_{n>0}$ be a sequence of ergodic measures with $P_{\mu_n}(\phi) \to P(\phi)$ and $\mu_n\rightarrow \mu$. 
 	By passing to a sub-sequence  we may assume $\mu_n({\rm BH}(\mu_1))=1$ for every $n>0$. 
 	Since $f$ is $\chi$-SPR for $\phi$, one has $f$ is $\chi$-SPR for $\phi$ on ${\rm BH}(\mu_1)$.
 	Then, by Theorem \ref{Thm: SPR-BH}, $\mu$ is an ergodic equilibrium state, and the Lyapunov exponents vary continuously.
 \end{proof}

The following corollary follows directly from Theorem \ref{Thm: B} and the definition of the SPR property on Borel homoclinic classes.
\begin{Corollary}
	Assume that $f: M\rightarrow M$ is a $C^{1+\alpha}$ surface diffeomorphism, and that $\phi: M \to \mathbb{R}$ is a H\"{o}lder continuous function such that there exists $\chi>0$ for which $f$ is $\chi$-SPR for $\phi$. 
	Then, for every ergodic equilibrium state $\mu$ of $\phi$, one has $f$ is $\chi$-SPR for $\phi$ on ${\rm BH}(\mu)$.
\end{Corollary}

\subsection{Strong positive recurrence symbolic dynamics}
Let $\mathcal{G}$ be a directed graph with finite or countable collection of vertices $\mathcal{V}$.
Without loss of generality, we assume that  $\# \{b\in \mathcal{V}:b\to a \}>0$ and $\# \{b\in \mathcal{V}:a \to b \}>0$ for any $a\in \mathcal{V}$.
Let 
$$\Sigma:=\Sigma(\mathcal{G}):=\{\bar{x}=(x_n)_{n\in \ZZ}\in \mathcal{V}^{\ZZ} : x_n \to x_{n+1},~\forall n \in \ZZ   \}.$$
We define the metric on $\Sigma$ by
$$d((x_n)_{n\in \ZZ},(y_n)_{n\in \ZZ}):= \exp \big(-\min \big\{|n|: x_n \neq y_n \big\}\big)~\text{and}~d((x_n)_{n\in \ZZ},(x_n)_{n\in \ZZ})=0.$$
Note that $\Sigma$ is locally compact iff $\mathcal{G}$ is \textit{locally finite} : $\# \{b\in \mathcal{V}:b\to a \}<\infty$ and $\# \{b\in \mathcal{V}:a \to b \}<\infty$ for any $a\in \mathcal{V}$.
In this section, we always assume that $\Sigma$ is locally compact.

Define the left shift map $\sigma: \Sigma \rightarrow \Sigma$, $\sigma(\bar{x})=\bar{y}$, where $y_n=x_{n+1}$ for any $n\in \NN$.
For each $a,b\in \mathcal{V}$,  we call that $a\overset{n}{\to}b$, if there are $\{c_1,\cdots,c_{n-1}\}\in \mathcal{V}$ such that $a\to c_1 \to \cdots \to c_{n-1} \to b$.
Then, $\sigma$ is topologically transitive if and only if  $\Sigma$ is \textit{irreducible} : for every $a,b\in \mathcal{V}$, there exists $n>0$ such that $a\overset{n}{\to}b$.  
Suppose that $\Sigma$ is irreducible, then $p:= {\rm gcd}\{n:a\overset{n}{\to}a \}$ does not depend on the choice of  $a$, and we call $p$ the \textit{period} of $\Sigma$.
If $p=1$, then we call $\Sigma$ is \textit{aperiodic}.
Note that $\sigma$ is topologically mixing if and only if  $\Sigma$ is irreducible and  aperiodic.
The following spectral decomposition theorem is well know, and the proof can be found in \cite[Lemma 6.4]{BCS25}.
\begin{Lemma}
	Let $\Sigma$ be an irreducible Markov shift with period $p\geq 1$.
	Then, there are disjoint closed subsets $\Sigma_0,\cdots,\Sigma_{p-1} \subset \Sigma$ such that $\Sigma=\bigcup_{i=0}^{p-1} \Sigma_i$, $\sigma(\Sigma_{i})=\Sigma_{(i+1)~{\rm mod}~p}$, and $\sigma^p: \Sigma_i \rightarrow \Sigma_i$  is topologically conjugate to a topologically mixing Markov shift for every $i=0,\cdots,p-1$.
\end{Lemma}

Let $V: \Sigma \rightarrow \RR$ be a continuous function. 
For $\beta>0$, we define
$$\|V\|_{\beta}:= \sup_{\bar{x}\in \Sigma} |V(\bar{x})|+\sup \left\{\frac{|V(\bar{x})-V(\bar{y})|}{d(\bar{x},\bar{y})^{\beta}},~\bar{x}\neq \bar{y},~\bar{x},\bar{y}\in \Sigma \right\}.$$
And we call $V$ is $\beta$-H\"{o}lder continuous, if $\|V\|_{\beta}<\infty$. 
Note that all \textit{H\"{o}lder continuous functions} on $\Sigma$ are bounded. 
Let $\phi: \Sigma \rightarrow \RR$ be a $\beta$-H\"{o}lder continuous function. 
For each $a\in \mathcal{V}$, we define 
\begin{align*}
	&Z_n(\phi,a):= \sum \big\{S_n\phi(\bar{x}):\bar{x}=(x_n)_{n\in \ZZ}\in \Sigma,~\sigma^n(\bar{x})=\bar{x},~x_0=x_{n}=a\big\},  \\
	&Z_n^{\ast}(\phi,a):= \sum \big\{S_n\phi(\bar{x}):\bar{x}=(x_n)_{n\in \ZZ}\in \Sigma,~\sigma^n(\bar{x})=\bar{x},~x_0=x_{n}=a,~x_{i}\neq a,~1\leq i\leq n-1\big\},
\end{align*}
where $S_n\phi=\phi+\phi\circ\sigma+\cdots+ \phi\circ \sigma^{n-1}$.
Let 
$$P_{\Sigma}(\phi):=\sup\left\{h_{\mu}(\sigma)+ \int \phi ~{\rm d}\mu: \mu~\text{is an invariant probability measure of}~\sigma \right\},$$
and let $P_{\Sigma}(\mu,\phi):=h_{\mu}(\sigma)+ \int \phi ~{\rm d}\mu$, where $\mu$ is a $\sigma$-invariant probability measure.

\begin{Theorem}\label{Thm:countable-pressure-bernoulli}
	Let $\Sigma$ be an irreducible Markov shift with period $p\geq 1$ and $\phi: \Sigma \rightarrow \RR$ be a H\"{o}lder continuous function. 
	Suppose that $P_{\Sigma}(\phi)<\infty$. 
	Then, we have 
	\begin{itemize}
		\item  for every $a\in \mathcal{V}$,
		           $P_{\Sigma}(\phi)=\limsup\limits_{n\rightarrow \infty} \frac{1}{n} \log Z_n(\phi,a)$;
		\item  if there is a $\sigma$-invariant measure $\mu$ such that $P_{\Sigma}(\mu,\phi)=P_{\Sigma}(\phi)$, then it is unique, and $(\Sigma,\sigma,\mu)$     measure theoretically isomorphic to the product of a Bernoulli scheme and a cyclic permutation of $p$ points.
	\end{itemize}
\end{Theorem}
We now explain the origin of Theorem~\ref{Thm:countable-pressure-bernoulli}. We call a $\sigma$-invariant measure $\mu$ the \textit{equilibrium state} of $\phi$ on $\Sigma$, if $P_{\Sigma}(\mu,\phi)=P_{\Sigma}(\phi)$.
The first item of Theorem~\ref{Thm:countable-pressure-bernoulli}, the variational principle of Gurevich pressure was proved in \cite[Theorem 3]{Sar99}.
Although it requires $(\Sigma,\sigma)$ to be topologically mixing, by the spectral decomposition theorem it also holds when $(\Sigma,\sigma)$ is topologically transitive.
The uniqueness of the equilibrium state of $\phi$ on $\Sigma$ was proved in \cite[Theorem 1.1]{BS03}.
The Bernoulli property of equilibrium state was established in \cite[Theorem 3.1, Lemma 4.1]{Sar11} or \cite[Theorem 1.1]{Da13}.

\smallskip

 The existence of an equilibrium state is equivalent to the positive recurrence property.
  
\begin{Definition}[Positive recurrence]
	Let $\Sigma$ be an irreducible Markov shift and $\phi: \Sigma \rightarrow \RR$ be a H\"{o}lder continuous function. 
	Suppose that $P_{\Sigma}(\phi)<\infty$.  
	We call that $\Sigma$ is  positive recurrence for $\phi$, if for some $a\in \mathcal{V}$ ( for all $a\in \mathcal{V}$)
	$$\sum_{n=1}^{\infty} e^{-n\cdot  P_{\Sigma}(\phi)}  Z_n(\phi,a)=\infty,~\sum_{n=1}^{\infty} e^{-n\cdot  P_{\Sigma}(\phi)} \cdot n \cdot Z_n^{\ast}(\phi,a)<\infty.$$
\end{Definition}

The following theorem was proved in \cite[Theorem 8]{Sar99} under the assumption that $(\Sigma,\sigma)$ is topologically mixing, and by the spectral decomposition theorem it also holds when $(\Sigma,\sigma)$ is topologically transitive.
Moreover, \cite{Sar99} provides a more detailed description of the equilibrium state in this case, including its Gibbs property and initial distribution.
\begin{Theorem}
		Let $\Sigma$ be an irreducible Markov shift and $\phi: \Sigma \rightarrow \RR$ be a H\"{o}lder continuous function. 
	    Suppose that $P_{\Sigma}(\phi)<\infty$.  
	    Then, an equilibrium state of $\phi$ exists if and only if  $~\Sigma$ is  positive recurrence for $\phi$.
\end{Theorem}

We now recall the notion of strong positive recurrence, which implies a series of statistical properties for the equilibrium state.
\begin{Definition}[Strong Positive recurrence]
	Let $\Sigma$ be an irreducible Markov shift and $\phi: \Sigma \rightarrow \RR$ be a H\"{o}lder continuous function. 
	Suppose that $P_{\Sigma}(\phi)<\infty$.  
	We call that $\Sigma$ is strong positive recurrence for $\phi$, if for some $a\in \mathcal{V}$ ( for all $a\in \mathcal{V}$)
	$$\Delta(\phi)=\Delta(\phi,a):=\limsup_{n\rightarrow +\infty} \frac{1}{n} \log Z_n(\phi,a) - \limsup_{n\rightarrow +\infty} \frac{1}{n} \log Z_n^{\ast}(\phi,a)>0$$
\end{Definition}
\begin{Remark}\label{Rem:SPRpro}
By the definition of $\Delta(\phi)$, if $\phi$ satisfies the SPR property,  then the SPR property automatically holds for any H\"{o}lder potential $\psi$ such that ${\rm Var}(\psi-\phi)<\Delta(\phi)$. 
Thus, for a $C^\infty$ surface diffeomorphism $f$, if $~{\rm Var}(\psi)<\Delta(0)$, then $f$ is SPR for $\psi$. 
In general, one only knows that $\Delta(0)$ is positive, but we do not know any explicit lower bound.
We thank O. Sarig for pointing out this statement to us.
\end{Remark}

By \cite{Jones62} or \cite[Lemma 7.2]{BCS25},  strong positive recurrence implies positive recurrence in the case $\phi\equiv 0$. 
One can similarly prove that strong positive recurrence implies positive recurrence for every H\"{o}lder continuous function $\phi:\Sigma \to \RR$; see also \cite{Sar01}.

Recall the definition of SPR for surface diffeomorphisms on Borel homoclinic classes given in Section \ref{SEC:SPR-B}.
\begin{Theorem}\label{Thm:SPRcoding}
	Let $f: M \to M$ be a $C^{1+\alpha}$ surface diffeomorphism and $\phi: M\rightarrow \RR$ be a H\"{o}lder continuous function.
   Assume that $X\subset M$ is a Borel homoclinic class.
   Then, for every $\chi_{0}>0$ there exists a locally compact irreducible Markov shift $\Sigma$ and a H\"{o}lder continuous map $\pi:\Sigma \rightarrow M$ such that
   \begin{enumerate}
   	\item [(a)] $\pi\circ \sigma = f\circ \pi$ and $\pi: \Sigma^{\#} \rightarrow M$ is finite to one, where $\Sigma^{\#}$ denotes the regular part of $\Sigma$, consisting of all  $(x_n)_{n\in \ZZ}\in \Sigma$ for which both $(x_i)_{i\geq 0},~(x_i)_{i\leq 0}$ have constant subsequence;
   	\item [(b)] for every ergodic $\sigma$-invariant measure $\hat{\mu}$ on $\Sigma$, the projection $\mu:=\pi_{\ast} \hat{\mu}$ is an ergodic $f$-invariant measure such that $\mu(X)=1$ and $ h_{\mu}(f)=h_{\hat{\mu}}(\sigma)$;
   	\item [(c)] for every ergodic $\chi_0$-hyperbolic measure $\mu$ with $\mu(X)=1$, there exists an ergodic $\sigma$-invariant measure $\hat{\mu}$ on $\Sigma$ such that $\mu=\pi_{\ast} \hat{\mu}$;
   	\item [(d)] there exists $\chi_1>0$ such that for every $\bar{x}\in \Sigma$, there is a splitting $T_{\pi(\bar{x})}M:=E^u(\bar{x})\oplus E^s(\bar{x})$ such that 
   	$\bar{x} \mapsto E^{u/s}(\bar{x})$ is H\"{o}lder continuous and $\limsup\limits_{n\rightarrow +\infty} \frac{1}{n}\log \|D_{\pi(\bar{x})}f^n|_{E^s(\bar{x})}\|<-\chi_1$, $\limsup\limits_{n\rightarrow +\infty} \frac{1}{n}\log \|D_{\pi(\bar{x})}f^{-n}|_{E^u(\bar{x})}\|<-\chi_1$;
   	\item [(e)] if $f$ is $\chi_2$-SPR for $\phi$ with $\chi_2>\chi_0$, then $\Sigma$ is SPR for $\hat{\phi}:=\phi\circ \pi$, $P_{X}(\phi)=P_{\Sigma}(\hat{\phi})$ and there exists $P_0< P_{X}(\phi)$ such that for every ergodic measure $\mu$ on $X$ with $P_{\mu}(\phi)>P_0$, one has $\mu(\pi(\Sigma^{\#}))=1$.
   \end{enumerate}
\end{Theorem}
We now explain the origin of Theorem~\ref{Thm:SPRcoding}. The SPR coding theorem was proved in \cite[Theorem 10.5]{BCS25}. 
For surface diffeomorphisms, Sarig \cite{Sar13} constructed a symbolic coding of non-uniformly hyperbolic sets by an countable topologically Markov shift; Buzzi-Crovisier-Sarig \cite{BCS22} constructed a symbolic coding of Borel homoclinic classes by an countable topologically irreducible Markov shift; and more recently, they \cite{BCS25}   showed that if a diffeomorphism is SPR on a homoclinic class, then it admits a symbolic coding by an SPR Markov shift.

\subsection{A list of statistical properties of equilibrium states}
In this section, we recall a collection of statistical properties of equilibrium states of diffeomorphisms, derived from the SPR coding.
These statistical properties were first stated in \cite{BCS25} for measures of maximal entropy.
To simultaneously deal with H\"{o}lder continuous potentials and geometric potentials $J^u(x):= \log |\det (D_xf|_{E^u(x)})|$, we use a notation of quasi-H\"{o}lder continuous functions as defined in \cite[Section 11]{BCS25}.

For a closed surface $M$, we denote $\PP TM$ the projective bundle of $TM$.
Let $X$ be a Borel homoclinic class.
A function $\phi:X\rightarrow \RR$ is said $\beta$-\textit{quasi-H\"{o}lder continuous}, if there exists a $\beta$-H\"{o}lder continuous function $\Phi: \PP TM \rightarrow \RR$ such that  $\Phi(x,E^u(x))=\phi(x)$ for all $x\in X$, and denote
$$ |\phi|_{\beta}:= \inf \left\{ \|\Phi\|_{\beta}:~\Phi(x,E^u(x))=\phi(x),~\forall x\in X \right\}.$$
If $\phi:M\rightarrow \RR$ is $\beta$-H\"{o}lder continuous, then $\phi |_{X}$ is  $\beta$-quasi-H\"{o}lder continuous and $|\phi|_X|_{\beta}\leq \|\phi\|_{\beta}$.
If $f:M\rightarrow M$ is a $C^{1+\alpha}$ surface diffeomorphism, then $x\to J^u(x)$ is $\alpha$-quasi-H\"{o}lder continuous (see \cite[Section 11.1]{BCS25} ).

Let $f:M\rightarrow M$ be a $C^{1+\alpha}$ surface diffeomorphism and let $X$ be a Borel homoclinic class.
Every quasi-Hölder continuous function on $X$ is measurable and bounded (see \cite[Lemma 11.5]{BCS25}).
Consider the irreducible Markov coding $\pi:\Sigma \rightarrow M$ as in Theorem \ref{Thm:SPRcoding}.
The following lemma was proved in \cite[Lemma 12.1]{BCS25}.
\begin{Lemma}\label{Lem:Lem3.9}
	For any $\beta>0$, there exist $C=C(\Sigma,\pi,\beta)$ and $\hat{\beta}=\hat{\beta}(\Sigma,\pi,\beta)$ such that for every $\beta$-quasi-H\"{o}lder continuous function $V:X\rightarrow \RR$, there is a $\hat{\beta}$-H\"{o}lder continuous function $\hat{V}:\Sigma \rightarrow \RR$ such that 
	$\hat{V}(\bar{x})= V \circ \pi(\bar{x})$ for every $\bar{x}\in \pi^{-1}(X)$ and $\| \hat{V}\|_{\hat{\beta}}\leq C \cdot |V|_{\beta}$.
\end{Lemma}

\subsubsection{Exponential mixing}\label{SEC:EM}
For SPR Markov shift with period $p\geq 1$, Buzzi-Crovisier-Sarig \cite[Theorem B.12]{BCS25} showed that every ergodic component of the unique equilibrium state with respect to $\sigma^p$ admits exponential decay of correlations :

\begin{Theorem}\label{Thm:C-EX}
	Let $\Sigma$ be an irreducible Markov shift of period $p\geq 1$ and $\phi: \Sigma \rightarrow \RR$ be a H\"{o}lder continuous function. 
	Suppose that $P_{\Sigma}(\phi)<\infty$ and $\Sigma$ is SPR for $\phi$.  
	Let $\hat{\mu}$ be the unique equilibrium state of $\phi$.
	Then, for each $\hat{\beta}>0$, there are $\hat{C}>0$ and $\hat{\theta}\in (0,1)$ such that for every $\hat{\beta}$-H\"{o}lder continuous functions $\hat{V}_1,\hat{V}_2:\Sigma\rightarrow \mathbb{R}$ and every ergodic component $\hat{\nu}$ of $\hat{\mu}$ with respect to $\sigma^p$, one has
	$$\left| \int \hat{V}_1(x) \cdot \hat{V}_2(\sigma^{pn}(x))~{\rm d}\hat{\nu}(x)- \int \hat{V}_1~{\rm d}\hat{\nu}  \cdot \int \hat{V}_2~{\rm d}\hat{\nu} \right|\leq \hat{C}\cdot \hat{\theta}^n \cdot \|\hat{V}_1\|_{\hat{\beta}} \cdot \|\hat{V}_2\|_{\hat{\beta}},~\forall n>0.$$
\end{Theorem}

     Now we give a stronger version of the proof of Theorem \ref{Thm: Exponential mixing}: the functions $V_1,V_2$ can be assumed to be $\beta$-quasi-H\"{o}lder continuous, rather than merely $\beta$-H\"{o}lder continuous.
     \begin{proof}[Proof of Theorem \ref{Thm: Exponential mixing}]
	 The proof is parallel to \cite[Section 12.1]{BCS25}.
	 Let $\mu$ be an ergodic equilibrium state of $\phi$ of mixing order $p\geq 1$, and set $X={\rm BH}(\mu)$.
	 Assume that $\phi: M\rightarrow \RR$ is H\"{o}lder continuous and $f$ is $\chi$-SPR for $\phi$.
 	 Fix $\chi_0<\chi$. 
	 By Theorem \ref{Thm:SPRcoding}, there exists a coding map $\pi: \Sigma \rightarrow X$, where $\Sigma$ is an irreducible Markov shift of period $q$ ($q$ may not be equal to $p$).
	 For the function $\phi$ we denote $\hat{\phi}=\phi \circ \pi$. 
	 Since $\pi$ is Hölder continuous, $\hat{\phi}$ is also H\"{o}der continuous.
     By assumption, $f$ is $\chi$-SPR for $X$ with respect to $\phi$, hence $\Sigma$ is SPR for $\hat{\phi}$  by Theorem \ref{Thm:SPRcoding} (e).
     Let $\hat{\mu}$ be the unique equilibrium state of $\hat{\phi}$. 
  
    \begin{Claim}
     The numbers $p,q\in \NN$ satisfies $(q/p)\in \NN$; the measure $\hat{\mu}$ satisfies $\pi_{\ast}(\hat{\mu})=\mu$;  and for every ergodic component $\nu$ of $\mu$ with respect to $f^{p}$, there exists an ergodic component $\hat{\nu}$ of $\hat{\mu}$ with respect to  $\sigma^{q}$ such that $\nu=\pi_{\ast}\hat{\nu}$.
    \end{Claim}
    \begin{proof}
     The proof is parallel to \cite[Lemma 12.2, 12.3]{BCS25}.
     By Theorem \ref{Thm:SPRcoding} (b) and (e), we have $P_{\Sigma}(\hat{\phi})=P_{X}(\phi)$. 
     Hence, $P_{\hat{\mu}}(\hat{\phi})=P_{X}(\phi)$ and $\pi_{\ast}(\hat{\mu})=\mu$.
     Since the mixing order of $\mu$ is $p$, there is an index function $C:X\rightarrow \{0,\cdots,p-1\}$ such that $C(f^n(x))=C(x)+(n~{\rm mod}~p)$.
     Let $\hat{C}=C\circ \pi$. 
     Then, $\hat{C} \circ \sigma^q(\bar{x}) = \hat{C}(\bar{x})+(q~{\rm mod}~p)$.
     For every ergodic  component $\hat{\nu}$ of $\hat{\mu}$ with respect to $\sigma^{q}$, since $(\Sigma,\sigma^q,\hat{\nu})$ is Bernoulli by Theorem~\ref{Thm:countable-pressure-bernoulli}, it follows that $\hat{\nu}$ is ergodic for $\sigma^{p\cdot q}$.  
     Moreover, we have $\hat{C} \circ \sigma^{p\cdot q}(\bar{x})=\hat{C}(\bar{x})$, so $\hat{C}(\bar{x})$ is constant $\hat{\nu}$-almost every where. 
     Consequently, for  $\hat{\nu}$-almost every $\bar{x}$ we have $\hat{C}(\bar{x})=\hat{C}(\bar{x})+(q~{\rm mod}~p)$, this implies that $(q~{\rm mod}~p)=0$ and so $(q/p)\in \NN$.
     
     Consider the ergodic decomposition of $\hat{\mu}=\frac{1}{q}\sum_{i=0}^{q-1}\sigma_{\ast}^{i} (\hat{\nu}')$ w.r.t. $\sigma^q$ and the ergodic decomposition of $\mu=\frac{1}{p}\sum_{j=0}^{p-1}f_{\ast}^{j} (\nu)$ w.r.t. $f^p$.
     For every $0 \leq j < p$, since $(X,f^p, f_{\ast}^{j} (\nu))$ is Bernoulli, it follows that $(X,f^p, f_{\ast}^{j} (\nu))$ is ergodic w.r.t. $f^{p\cdot q}$.
     Similarly, $(\Sigma, \sigma^q, \sigma_{\ast}^{i} (\hat{\nu}'))$ is ergodic w.r.t. $\sigma^{p\cdot q}$ for every $0 \leq i < q$.
     Note that $\mu=\frac{1}{p}\sum_{j=0}^{p-1}f_{\ast}^{j} (\nu)$ and $\mu= \frac{1}{q}\sum_{i=0}^{q-1}\pi_{\ast} \circ \sigma_{\ast}^{i} (\hat{\nu}') $ are two ergodic decomposition of $\mu$  w.r.t. $f^{p\cdot q}$.
     Hence, there exists $i\in \{0,\cdots q-1\}$ such that $\hat{\nu}=\pi_{\ast}(\sigma_{\ast}^{i} (\hat{\nu}'))$, where $\hat{\nu}:=\sigma_{\ast}^{i} (\hat{\nu}')$ is an ergodic component of $\hat{\mu}$ w.r.t. $\sigma^q$.
    \end{proof}

    For each $\beta>0$, Lemma \ref{Lem:Lem3.9} provides $C_{\beta}>0$ and $\hat{\beta}>0$ such that for every $\beta$-quasi-H\"{o}der continuous $V:X\to\RR$, there exists a $\hat{\beta}$-H\"{o}der continuous function $\hat{V}:\Sigma\to\RR$ satisfying
    $\hat{V}(\bar{x})= V \circ \pi(\bar{x})$ for every $\bar{x}\in \pi^{-1}(X)$ and $\| \hat{V}\|_{\hat{\beta}}\leq C_{\beta} \cdot |V|_{\beta}$.
    Let $\nu$ be an ergodic component of $\mu$ with respect to $f^{p}$. 
    Then, by Claim, there exists an ergodic component $\hat{\nu}$ of $\hat{\mu}$ with respect to $\sigma^{q}$ such that $\nu=\pi_{\ast}\hat{\nu}$.
    Since $\mu(\pi(\Sigma))=1$, the same holds for $\nu$.
    Therefore, by Theorem \ref{Thm:C-EX}, there exist $\hat{C}_{\hat{\nu}}>0$ and $\hat{\theta}_{\hat{\nu}}\in(0,1)$ such that for every pair of $\beta$-quasi-H\"{o}der continuous functions $V_1,V_2:X\to\RR$ and every $n>0$, one has
    \begin{align*}
    	        &\left| \int V_1(x) \cdot V_2(f^{qn}(x))~{\rm d}\nu(x)- \int V_1~{\rm d}\nu  \cdot \int V_2~{\rm d}\nu \right| \\
    	     =~& \left| \int \hat{V}_1(x) \cdot \hat{V}_2(\sigma^{qn}(x))~{\rm d}\hat{\nu}(x)- \int \hat{V}_1~{\rm d}\hat{\nu}  \cdot \int \hat{V}_2~{\rm d}\hat{\nu} \right| \\
    	     \leq ~& \hat{C}_{\hat{\nu}}\cdot \hat{\theta}_{\hat{\nu}}^n  \cdot \|\hat{V}_1\|_{\hat{\beta}} \cdot \|\hat{V}_2\|_{\hat{\beta}} \\
    	      \leq~& C_{\beta}^2 \cdot \hat{C}_{\hat{\nu}}\cdot \hat{\theta}_{\hat{\nu}}^n \cdot |V_1|_{\beta} \cdot |V_2|_{\beta}.
    \end{align*}
    Then, for each $n>0$, write $pn=qk_n+r_n$, where $0\leq r_n<q$ and $(r_n/p)\in \NN$ (by $(q/p)\in \NN$).
    Note that $\nu$ is $f^p$ ergodic, we have $\int V_2~{\rm d}\nu=\int (V_2\circ f^{r_n})(x)~{\rm d}\nu(x)$ and 
    $$|V_2\circ f^{r_n}(x)|_{\beta}\leq (\sup_{x\in M} \|D_xf^r_n\| )^{\beta} \cdot |V_2|_{\beta}\leq  (\sup_{x\in M} \|D_xf\|)^{q\cdot \beta} \cdot |V_2|_{\beta}.$$    
    Therefore, we have 
    \begin{align*}
    &~\left| \int V_1(x) \cdot V_2(f^{pn}(x))~{\rm d}\nu(x) - \int V_1(x)~{\rm d}\nu(x)  \cdot \int V_2(x)~{\rm d}\nu(x) \right| \\
    \leq &~\left |\int V_1(x) \cdot (V_2\circ f^{r_n})(f^{qk_n}(x))~{\rm d}\nu(x)- \int V_1(x)~{\rm d}\nu(x)  \cdot \int V_2\circ f^{r_n}(x)~{\rm d}\nu(x)  \right|,\\
    \leq &~ C_{\beta}^2 \cdot \hat{C}_{\hat{\nu}}  \cdot \hat{\theta}_{\hat{\nu}}^{(p/q)\cdot n} \cdot |V_1|_{\beta} \cdot |V_2\circ f^{r_n}(x)|_{\beta}\\
    \leq &~ C_{\beta}^2 \cdot \hat{C}_{\hat{\nu}} \cdot (\sup_{x\in M} \|D_xf\|)^{q\beta} \cdot  \hat{\theta}_{\hat{\nu}}^{(p/q)\cdot n} \cdot |V_1|_{\beta} \cdot |V_2|_{\beta}.
    \end{align*}   
    For each $i=1,2$, if  $V_i:M\rightarrow \RR$ is H\"{o}lder continuous, we have $|V_i|_{\beta} \leq \|V_i\|_{\beta}$. 
    Since $\hat{\mu}$ has only finitely many ergodic components, it suffices to take $C$ larger than all $C_{\beta}^2 \cdot \hat{C}_{\hat{\nu}} \cdot (\sup \|D_xf\|)^{q\beta}$ and $\theta$ larger than all $\hat{\theta}_{\hat{\nu}}^{(p/q)}$ to complete the proof of Theorem \ref{Thm: Exponential mixing}.
\end{proof}

\begin{proof}[Proof of Theorem \ref{Thm:T-ET}]
    Let  $f: M\rightarrow M$ be a topologically mixing $C^{\infty}$ surface diffeomorphism with positive entropy.  
	Assume that $\phi: M \rightarrow \mathbb{R}$ is H\"{o}lder continuous with ${\rm Var}(\phi)< h_{\rm top}(f)$.
	Since $f:M \to M$ is topologically mixing, by \cite[Theorem 6.12]{BCS22} there exists at most one Borel homoclinic class that can support an ergodic measure with positive entropy, and the period of this homoclinic class is $1$.
	By Lemma \ref{Lem: EH}, there exists $\delta>0$ such that every ergodic equilibrium state $\mu$ of $\phi$ satisfies $h_{\mu}(f)>\delta$.
    Then, by Lemma \ref{Lem: HES}, $f$ admits a unique equilibrium state $\mu$ of $\phi$. 
    Moreover, by \cite[Corollary 3.3]{BCS22}, the mixing order of this equilibrium state is $1$.
	Together with Theorem \ref{Thm: A} and Theorem \ref{Thm: Exponential mixing}, we complete the proof of Theorem \ref{Thm:T-ET}.
\end{proof}

\subsubsection{Effective intrinsic ergodicity} \label{SEC:EIE}
     R\"{u}hr-Sarig \cite[Theorem 7.1 and Section 9]{Sar22} proved effective intrinsic ergodicity for SPR Markov shifts.
\begin{Theorem}\label{Thm:C-EIE}
	 Let $\Sigma$ be an irreducible Markov shift of period $p\geq 1$ and $\phi: \Sigma \rightarrow \RR$ be a $\hat{\alpha}$-H\"{o}lder continuous function. 
	 Suppose that $P_{\Sigma}(\phi)<\infty$ and $\Sigma$ is SPR for $\phi$.  
	 Let $\hat{\mu}$ be the unique equilibrium state of $\phi$.
	 Then, for each $\hat{\beta}\in (\hat{\alpha},1)$, there are $\hat{C}:=\hat{C}(\hat{\alpha},\hat{\beta})>0$ such that for every $\hat{\beta}$-H\"{o}lder continuous functions $\hat{V}:\Sigma\rightarrow \mathbb{R}$ and every $\sigma$-invariant measure $\hat{\nu}$ on $\Sigma$, one  has
	 $$\left| \int \hat{V}(x)~{\rm d}\hat{\nu}(x)-\int \hat{V}(x)~{\rm d}\hat{\mu}(x) \right| \leq \hat{C} \cdot  \|\hat{V}\|_{\hat{\beta}} \cdot  \sqrt{P_{\Sigma}(\phi)-P_{\nu}(\phi)}.$$
\end{Theorem}
    Now we give the proof of effective intrinsic ergodicity on an SPR Borel homoclinic class.
\begin{Theorem}\label{Thm: EIE-BH}
	Assume that $f$ is $\chi$-SPR for a H\"{o}lder continuous function $\phi$ on a Borel homoclinic class $X$. 
	Let $\mu$ be the unique equilibrium state on $X$.
	For every $\beta\in (0,1)$, there is $C>0$ such that for every $\beta$-quasi-H\"{o}lder continuous function $V:M\rightarrow \mathbb{R}$ and every invariant measure $\nu$ with $\nu(X)=1$,
	$$\left| \int V (x) ~{\rm d}\mu(x)- \int V (x)~{\rm d}\nu(x)  \right|\leq C \cdot |V|_{\beta} \cdot \sqrt{P_{\mu}(\phi)-P_{\nu}(\phi) }.$$
	Moreover, we also have  $\left| \lambda^{+}(\mu)-\lambda^{+}(\nu) \right|\leq C\sqrt{P_{\mu}(\phi)-P_{\nu}(\phi)}$.
\end{Theorem}

\begin{proof}
			Assume that $\phi: M\rightarrow \RR$ is $\alpha_0$-H\"{o}lder continuous and $f$ is $\chi$-SPR for $\phi$ on $X$.
			Fix $\chi_0<\chi$. 
			By Theorem \ref{Thm:SPRcoding}, there exists a coding map $\pi: \Sigma \rightarrow X$, where $\Sigma$ is an irreducible Markov shift of period $p$.
			Then, there exists $\hat{\alpha}_1\in (0,1)$ such that $\hat{\phi}=\phi \circ \pi$ is $\hat{\alpha}_1$-quasi-H\"{o}der continuous.			
			For each $\beta>0$, Lemma \ref{Lem:Lem3.9} provides $C_{\beta}\geq 1$ and $\hat{\beta}>0$ such that for every $\beta$-quasi-H\"{o}der continuous $V:X\to\RR$, there exists a $\hat{\beta}$-H\"{o}der continuous function $\hat{V}:\Sigma\to\RR$ satisfying
			$\hat{V}(\bar{x})= V \circ \pi(\bar{x})$ for every $\bar{x}\in \pi^{-1}(X)$ and $\| \hat{V}\|_{\hat{\beta}}\leq C_{\beta}|V|_{\beta}$.
			Take $0<\hat{\alpha}<\min\{\hat{\alpha}_1,\hat{\beta}\}$, then $\hat{\phi}=\phi \circ \pi$ is also $\hat{\alpha}$-quasi-H\"{o}der continuous.
			
			By Theorem \ref{Thm:SPRcoding} (e) and (c), there is $P_0<P_{X}(\phi)$ such that for every ergodic measure $\nu$ with $\nu(X)=1$ and $P_{\nu}(\phi)>P_0$, there exists an ergodic $\sigma$-invariant measure $\hat{\nu}$ on $\Sigma$ such that $\pi_{\ast}(\hat{\nu})=\nu$.
			Therefore, by Theorem \ref{Thm:C-EIE}, there exists $\hat{C}=\hat{C}(\hat{\alpha},\hat{\beta})>0$ such that for every $\beta$-quasi-H\"{o}der continuous $V:X\to\RR$
			\begin{align*}
			 	\left| \int {V}~{\rm d}{\nu}-\int {V}~{\rm d}{\mu} \right|&=\left| \int \hat{V}~{\rm d}\hat{\nu}-\int \hat{V}~{\rm d}\hat{\mu} \right| \\
			 	 &\leq \hat{C} \cdot  \|\hat{V}\|_{\hat{\beta}} \cdot  \sqrt{P_{\Sigma}(\hat{\phi})-P_{\hat{\nu}}(\hat{\phi}})\\
			 	 &\leq \hat{C} \cdot C_{\beta} \cdot  |V|_{\beta} \cdot  \sqrt{P_{\mu}(\phi)-P_{\nu}(\phi)}.
			\end{align*}
			If $P_{\nu}(\phi)\leq P_0$, then we have
			 $$\left| \int {V}~{\rm d}{\nu}-\int {V}~{\rm d}{\mu} \right|\leq 2|V|_{\beta}\leq 2(P(\phi)-P_0)^{-1/2} \cdot |V|_{\beta} \cdot  \sqrt{P_{\mu}(\phi)-P_{\nu}(\phi)}.$$
		    Let $C:=\max\{\hat{C}\cdot C_{\beta},  2(P(\phi)-P_0)^{-1/2}\}$, then for every ergodic measure $\nu$ with $\nu(X)=1$, we have 
			 $$\left| \int {V}~{\rm d}{\nu}-\int {V}~{\rm d}{\mu} \right| \leq  C \cdot  |V|_{\beta} \cdot  \sqrt{P_{\mu}(\phi)-P_{\nu}(\phi)}.$$
			If $\nu$ is non-ergodic, then by the ergodic decomposition the above inequality also holds.
			Since $J^u$ is $\alpha$-quasi holder continuous, we may take $C$ large enough such that 
			 $\left| \lambda^{+}(\mu)-\lambda^{+}(\nu) \right|\leq C\sqrt{P_{\mu}(\phi)-P_{\nu}(\phi) }$ 
			for every invariant measure on $X$. 
			This completes the proof.
			 \end{proof}
             
             \begin{proof}[Proof of Theorem \ref{Thm: Effective intrinsic ergodicity}]
             	Assume that $\phi: M\rightarrow \RR$ is H\"{o}lder continuous and $f$ is $\chi$-SPR for $\phi$ on $M$.
             	\begin{Claim}
             		For every sequence of ergodic measures $\{\mu_n\}$ with $P_{\mu_n}(\phi) \rightarrow P(\phi)$, there exists $N>0$ such that for every $n>N$, there is an ergodic equilibrium state $\nu_n$ of $\phi$ such that $\mu_n \overset{h}{\sim} \nu_n$. 
             	\end{Claim}
             	\begin{proof}[Proof of Claim]
             	Suppose for contradiction that there exists a sequence of ergodic measures $\{\mu_n\}$ with 
             	$P_{\mu_n}(\phi) \rightarrow P(\phi)$ such that $\mu_n$ not homoclinic related with any ergodic equilibrium state of $\phi$ for every $n>0$.
                By the argument in the proof of Theorem \ref{Thm: B}, there exists a subsequence $\{\mu_{n_k}\}_{k>0}$ such that $\mu_{n_{k_1}}\overset{h}{\sim} \mu_{n_{k_2}}$ for every $k_1,k_2>0$, and $\mu_{n_k} \to \mu$ as $k\to\infty$.
             	Therefore, by Theorem \ref{Thm: SPR-BH}, we have $\mu$ is an ergodic equilibrium state of $\phi$ and $\mu \overset{h}{\sim}  \mu_{n_{k}}$ for every $k>0$. 
             	This is a contradiction.
             	\end{proof}
             	
             Let $\{\mu_1,\cdots,\mu_N\}$ be the ser of all ergodic  equilibrium states of $\phi$.
             By the claim, there exists $P_1\in (0,P(\phi))$ such that for every ergodic measure $\nu$ with $P_{\nu}(\phi)>P_1$, there exists an ergodic equilibrium state $\mu_j$ such that $\mu_j \overset{h}{\sim}  \nu$, and so $\nu({\rm BH}(\mu_j))=1$. 
             By Theorem \ref{Thm: EIE-BH}, for every $\beta\in (0,1)$, there exists $C_j>0$ such that for every $\beta$-quasi H\"{o}lder continuous function $V:M\rightarrow \RR$, one has
             	$$\left| \int V  ~{\rm d}\mu_j- \int V ~{\rm d}\nu  \right|\leq C_j \cdot |V|_{\beta} \cdot \sqrt{P_{\mu_j}(\phi)-P_{\nu}(\phi) }.$$
             Take $C'=\max \{C_1,\cdots,C_N\}$, Theorem \ref{Thm: Effective intrinsic ergodicity} holds for every ergodic measure $\nu$ with $P_{\nu}(\phi)>P_1$.
             For every ergodic measure $\nu$ with $P_{\nu}(\phi)\leq P_1$, fix an ergodic equilibrium state $\mu_0\in \{\mu_1,\cdots,\mu_N \}$. Then, 
             $$\left| \int V  ~{\rm d}\mu_0- \int V ~{\rm d}\nu  \right|\leq 2|V|_{\beta}\leq 2(P(\phi)-P_1)^{-1/2} \cdot |V|_{\beta} \cdot  \sqrt{P_{\mu_0}(\phi)-P_{\nu}(\phi)}.$$
             Take $C>\max\{C',  2(P(\phi)-P_1)^{-1/2}\}$, then Theorem \ref{Thm: Effective intrinsic ergodicity} holds for every ergodic measure $\nu$.
             
             Now assume that $\nu$ is an invariant measure (not necessary ergodic), and consider its ergodic decomposition $\nu=\int \nu_x ~{\rm d}\nu(x)$. 
             Set $E_0:=\{x: P_{\nu_x}(\phi)\leq P_1 \} $ and $E_j:=\{x:P_{\nu_x}(\phi)>P_1,~\nu_x\overset{h}{\sim}\mu_j \}$, $1\leq j \leq N$.
             Let $\mu= \sum_{j=0}^{N} \nu(E_j) \cdot \mu_j$. 
             Then,  $\mu$ is an equilibrium state of $\phi$. 
             Moreover, for every $\beta$-quasi-H\"{o}der continuous function $V:X\to\RR$, we have
             \begin{align*}
             	\left| \int V ~{\rm d}\mu- \int V~{\rm d}\nu  \right|&\leq  \sum_{j=0}^{N} \int  \chi_{E_j}(x)\cdot \left| \int V  ~{\rm d}\nu_x-\int V ~{\rm d}\mu_j \right| ~{\rm d}\nu(x) \\
             	&\leq C\cdot |V|_{\beta} \cdot \Big(\int \sum_{j=0}^{N} \sqrt{P_{\mu_j}(\phi)-P_{\nu_x}(\phi)} \cdot \chi_{E_j}(x) ~{\rm d}\nu(x) \Big)\\
             	&\leq C\cdot |V|_{\beta} \cdot   \Big(\int \sum_{j=0}^{N}\big(P_{\mu_j}(\phi)-P_{\nu_x}(\phi) \big)\cdot \chi_{E_j}(x) ~{\rm d}\nu(x) \Big)^{1/2}\\
             	&\leq C\cdot |V|_{\beta} \cdot  \sqrt{P_{\mu}(\phi)-P_{\nu}(\phi)}.             	
             \end{align*}
             If $V:M\rightarrow \RR$ is $\beta$-H\"{o}lder continuous, then we have $|(V |_X)|_{\beta}\leq \|V\|_{\beta}$.  
             This completes the proof of Theorem \ref{Thm: Effective intrinsic ergodicity}.
             \end{proof}
\begin{proof}[Proof of Corollary \ref{Cor:EIE}]
    By Theorem \ref{Thm: Effective intrinsic ergodicity}, it suffices to show that there exists a unique equilibrium state of $\phi$.
    Since $f:M \to M$ is topologically transitive, by \cite[Theorem 6.12]{BCS22}, there exists at most one Borel homoclinic class that can support an ergodic measure with positive entropy.
    By Lemma \ref{Lem: EH}, there exists $\delta>0$ such that every ergodic equilibrium state of $\phi$ has metric entropy greater than $\delta$.
    Therefore, by Theorem \ref{Thm:FiniteBH}, the ergodic equilibrium state is unique.
\end{proof}

 \section*{Acknowledgments}
We would like to express our sincere gratitude to J. Buzzi and O. Sarig for sharing valuable insights with us and for offering new perspectives on our work.
We also thank the Tianyuan Mathematical Center in Northeast China for its support.

\flushleft{\bf Chiyi Luo}  \\
\small School of Mathematics and Statistics, Jiangxi Normal University, 330022 Nanchang, P. R. China.\\
\textit{E-mail:} \texttt{luochiyi98@gmail.com}\\

\flushleft{\bf Dawei Yang}  \\
\small School of Mathematical Sciences,  Soochow University, 215006 Suzhou, P. R. China.\\
\textit{E-mail:} \texttt{yangdw@suda.edu.cn}\\

\end{document}